\documentclass[12pt,reqno]{amsart}

\usepackage{amsmath,amsthm,amsfonts,amssymb,amscd,  multicol, verbatim, enumerate, stackrel}
\usepackage{verbatim}
\usepackage{fullpage}
\input{xy}
\usepackage{xcolor}
\usepackage[english]{babel}
\usepackage{cite}
\usepackage{amsfonts}
\usepackage{latexsym}
\usepackage{amsthm}
\usepackage{amsopn}
\usepackage{amsmath}
\usepackage[all]{xy}

\usepackage{hyperref}

\usepackage{verbatim}
\usepackage{amssymb}
\usepackage{mathrsfs}
\usepackage{float}
\usepackage{fullpage}
\usepackage{amscd}
\usepackage{amscd}
\usepackage{color}

\usepackage{graphicx}
\usepackage{tikz}

\theoremstyle{plain}
\newtheorem*{thm*}{Theorem}
\newtheorem*{con*}{Conjecture}
\newtheorem{thm}{Theorem}[section]
\newtheorem{pro}[thm]{Proposition}

\newtheorem{lem}[thm]{Lemma}
\newtheorem{con}[thm]{Conjecture}

\newtheorem{cor}[thm]{Corollary}

\newtheorem*{str*}{Strategy}

\theoremstyle{definition}
\newtheorem{defn}[thm]{Definition}

\theoremstyle{remark}
\newtheorem{rmk}[thm]{Remark}

\title[Flipclasses and Combinatorial Invariance]{Flipclasses and Combinatorial Invariance for Kazhdan--Lusztig polynomials}
\author[Esposito, Marietti]{Francesco Esposito, Mario Marietti}

\date{}

\address{Francesco Esposito, Dipartimento di Matematica, Universit\`a degli Studi di Padova,
via Trieste 63, 35121 Padova, Italy}
\address{Mario Marietti, Dipartimento  di Ingegneria Industriale e Scienze Matematiche, Universit\`a Politecnica delle Marche, via Brecce Bianche, 60131 Ancona,  Italy}

\email{esposito@math.unipd.it}
\email{m.marietti@univpm.it}


\begin{document}

\maketitle

\begin{abstract}
In this work,  we investigate a novel approach to the Combinatorial Invariance Conjecture of Kazhdan--Lusztig polynomials for the symmetric group.  Using the new concept of flipclasses, we introduce some combinatorial invariants of intervals in the symmetric group whose analysis leads us to a recipe to compute the coefficients of $q^h$ of the Kazhdan--Lusztig $\widetilde{R}$-polynomials, for $h\leq 6$. This recipe depends only on  the isomorphism class (as a poset) of the interval indexing the polynomial and thus  provides new evidence for the Combinatorial Invariance Conjecture.
\end{abstract}

\section{Introduction}
The introduction of Kazhdan--Lusztig polynomials by Kazhdan and Lusztig in \cite{K-L} has proved to be a major paradigm shift in Lie theory and representation theory, giving a fundamental contribution to the development of geometric representation theory. 
These polynomials, indexed by two elements  $u$ and $v$ in a Coxeter group $W$,  have soon found applications in several contexts. 

Among the combinatorial problems regarding Kazhdan--Lusztig polynomials, arguably the most prominent one is what is usually referred to as the Combinatorial Invariance Conjecture. It was independently posed in the 80's by Lusztig, in private, and Dyer (see \cite[Remark~7.31]{Dyeth}).

\begin{con}
\label{comb-inv-con}
The Kazhdan--Lusztig polynomial $P_{u,v}(q)$ depends only on the isomorphism type of the interval $[u,v]$ as a partially ordered set under Bruhat order.
\end{con}
\noindent 
Although widely open, the Combinatorial Invariance Conjecture has been settled in some special cases, among which: 
 intervals of rank $\leq 4$ (see, e.g., \cite[7.31]{Dyeth}, and \cite[Chapter~5, Exercises~7 and~8]{BB}), intervals of rank $\leq 8$ in type $A$ and $\leq 6$ in types $B$ and $D$ (\cite{Inc}), 
   intervals in type $\tilde{A}_2$ (\cite{BLP}), intervals which are lattices (\cite[7.23]{Dyeth}, \cite{Bre94}),  and intervals starting from the identity  (\cite{BCM1}).
  We refer the reader to \cite{BreOPAC} for further details on  
 Conjecture \ref{comb-inv-con}.
Also a parabolic version of  Conjecture \ref{comb-inv-con} holds for intervals starting from the identity (\cite{Mtrans}, \cite{M}).
On the other hand,  certain entirely poset-theoretical generalizations of Conjecture \ref{comb-inv-con} fail (\cite{BCM2}, \cite{MJaco}).

The Combinatorial Invariance Conjecture for Kazhdan--Lusztig polynomials is equivalent to the Combinatorial Invariance Conjecture for $R$-polynomials (also introduced in \cite{K-L}), as well as for 
 $\widetilde{R}$-polynomials, which are  useful rescalings of  $R$-polynomials. As a matter of fact, some of the above results have been achieved by means of these other families of polynomials.

Recently, an even broader interest in the Combinatorial Invariance Conjecture has been sparked by the use of certain machine learning models to attack the problem (see 
\cite{BBDVW}, \cite{DVBBZTTBBJLWHK}). Blundell, Buesing, Davies, Veli\u ckovi\'c, and Williamson discovered a formula for the Kazhdan--Lusztig polynomials of the symmetric group, and propose a conjecture that implies
the Combinatorial Invariance Conjecture for the symmetric group. This conjecture is based on the concept of a hypercube decomposition, which has already been the object of several studies (see \cite{B-G}, \cite{B-M}, \cite{G-W}).

In this paper, we approach the Combinatorial Invariance Conjecture for the symmetric group $\mathfrak S_n$ from a new perspective.
We believe that the intervals might not be the ultimate atoms for the combinatorial invariance and should be decomposed into smaller building blocks, viz. the {\em flipclasses}. The guiding principle is that, regarding the combinatorial invariance,  all concepts  associated with intervals should in fact be viewed  as shadows of finer concepts associated with flipclasses. This work is intended to give first evidence in support of this new approach.

Let $u,v \in \mathfrak S_n$ with $u\leq v$ in Bruhat order. Let $B(\mathfrak S_n)$  denote  the Bruhat graph of $\mathfrak S_n$.
For each $h \in \mathbb N$, we partition the set  of paths from $u$ to $v$ of length $h$ in $B(\mathfrak S_n)$ into orbits under the action of the group generated by the flips. We call the orbits of this group {\em $h$-flipclasses}.
The multiset of the unlabelled $h$-flipclasses of an interval is a combinatorial invariant. By a combinatorial invariant of an interval we mean an object that depends only on the isomorphism class of the interval as a poset. With each flipclass $F$, we associate  two graphs, the support graph $S_F$ and the time-support graph $TS_F$, and a polynomial, the $\iota$-polynomial $\iota_F$  (see Section~\ref{definizioni} for the definitions). These new concepts yield combinatorial invariants.

We study these new combinatorial objects and prove that, if two flipclasses have the same $\iota$-polynomial, then they also have the same number of increasing paths (w.r.t. any reflection ordering) of length $h$, for $h\leq 6$. Thus, the multiset of the  $\iota$-polynomials of the $h$-flipclasses of paths from $u$ to $v$ 
determines the coefficient of $q^h$ of the $\widetilde{R}$-polynomial  $\widetilde{R}_{u,v}(q)$, for $h\leq 6$. This implies that these coefficients are combinatorial invariants. We stress the facts that we find an explicit algorithm for computing such coefficients of the $\widetilde{R}$-polynomial  $\widetilde{R}_{u,v}(q)$ starting from the isomorphism class of Bruhat interval $[u, v]$ as a poset and  that this algorithm works for any interval in any symmetric group.

Let us succintly describe the ingredients of the proof. 
Fix $h$ in $\mathbb N$. 
The set whose elements are the $h$-flipclasses of any $\mathfrak S_n$
(letting $n$ vary)
is clearly infinite. On the other hand, we prove that the number of isomorphism classes of $h$-flipclasses is finite and that all such isomorphism classes can be constructed from $h$-flipclasses of $\mathfrak S_{h+1}$. Thus, we reduce the infinity of the general case to a finite analysis. By definition,  isomorphic flipclasses have the same number of increasing paths  w.r.t. the lexicographic order. But this alone is not applicable to the Combinatorial Invariance Conjecture because the multiset of the isomorphism classes of the $h$-flipclasses of an interval is not a combinatorial invariant of the interval since the flipclasses come with the labels of the edges.
However, we show that, for $h\leq 6$, the $\iota$-polynomial detects the number of increasing paths, and we construct the map sending a $\iota$-polynomial to the number of increasing paths  of any flipclass having that $\iota$-polynomial. These computations can be done with the aid of a computer thanks to the above reduction to finite calculations, but unfortunately they become extremely unwieldy and we carry out this program  up to  $h=6$. 
The multiset of the $\iota$-polynomials  of  the $h$-flipclasses of a fixed interval $[u,v]$ is a combinatorial invariant. Hence, the number of increasing paths from $u$ to $v$ of length $h$, for $h\leq 6$, is a combinatorial invariant: by a result of Dyer  (see \cite{DyeComp}),   this number is equal to the coefficient of $q^h$ of the $\widetilde{R}$-polynomial  $\widetilde{R}_{u,v}(q)$.

The algorithm derived from our study to compute the coefficient of $q^h$, for $h\leq 6$, of the $\widetilde{R}$-polynomial  $\widetilde{R}_{u,v}(q)$ from the isomorphism class of $[u,v]$ proceeds along the following steps.

\begin{enumerate}
\item Construct the $h$-flipclasses of paths from $u$ to $v$.
\item For each of these flipclasses, compute its $\iota$-polynomial.
\item
\label{numeri} For each of these flipclasses, compute its number of increasing paths starting  from its $\iota$-polynomial.
\item Sum the numbers obtained in the previous step to obtain  the coefficient of $q^h$ of $\widetilde{R}_{u,v}(q)$ 
\end{enumerate}
To give an idea of the rate of growth, we note that  the sequence of the number of $h$-flipclasses of  $\mathfrak S_{h+1}$   is $(4, 4, 50, 1096, 36634, 1701056)$, for $h=1, \ldots  ,6 $.

\medskip
It is well known that an interval has the same number of increasing paths w.r.t. any reflection ordering.  We prove that the same is true for flipclasses. This is first evidence  in support of the principle that the appropriate setting in formulating the Combinatorial Invariance Conjecture could be that of flipclasses. Differently from the proofs in the literature of the analogous result for intervals, our proof is local in nature. It thus sheds new light also on the case of intervals.

\medskip 
A natural conjecture  stems from our study, in line with the underlying philosophy. It is a refinement of the Combinatorial Invariance Conjecture and predicts that if two unlabelled $h$-flipclasses are isomorphic, then they have the same number of increasing paths.

 \medskip
 The organization of the paper is as follows. In Section~2, we recall some definitions, notation, and results that are needed in the paper. In Section~3, we introduce the main new concepts of this work, namely support graphs, time-support graphs, $\iota$-polynomials, and flipclasses. In Section~4, we prove that a flipclass has the same number of increasing paths w.r.t. any reflection ordering. In Section~5, we give some results on flipclasses, their support and time-support graphs, and their $\iota$-polynomials. In Section~6, we apply results of Section~5  to reduce the infinite computations to a finite number, for all $h\in \mathbb N$.   In Section~7, we provide a general strategy to prove the combinatorial invariance of the coefficient of $q^h$ of the $\widetilde{R}$-polynomials., for all $h \in \mathbb N$, and apply it for $h\leq 6$.  In Section~8, we propose a conjecture that implies the Combinatorial Invariance Conjecture.

\section{Preliminaries}

This section reviews the background material. 

Given $n\in \mathbb N^+$, we denote the set $\{1,2, \ldots, n\}$ by $[n]$ and the symmetric group on $[n]$ by $\mathfrak S_n$.   We write the elements of $\mathfrak S_n$ in one-line notation (so $u=[a_1, \ldots , a_n]$, or simply $u=a_1 \cdots a_n$, means
$u(i)=a_i$ for all $i \in [n]$) as well as in disjoint cycle-form, omitting to write the $1$-cycles.
We view $\mathfrak S_n$ as a Coxeter group with the simple transpositions $(i,i+1)$, for $i \in [n-1]$, as Coxeter generators. It is a Coxeter group of type $A$. 
We fix our notation on the symmetric group in the following list:
\smallskip 

$
\begin{array}{@{\hskip-1.3pt}l@{\qquad}l}
e &  \textrm{identity of $\mathfrak S_n$}, 
\\
S &= \{ (i,i+1) : i \in [n-1] \},  \textrm{  the set of {\em simple transpositions} of $\mathfrak S_n$},
\\
\ell  &  \textrm{the length function of $\mathfrak S_n$ with respect to $S$},
\\
w_0 &  \textrm{ the longest   element of $\mathfrak S_n$},
\\
T &= \{ w s w ^{-1} : w \in \mathfrak S_n, \; s \in S \},  \textrm{  the set of {\em transpositions} of $\mathfrak S_n$},
\\
D_R(w) & =\{ s \in S : \; \ell(w  s) < \ell(w ) \},  \textrm{ the right descent set of $w\in \mathfrak S_n$},
\\
D_L(w) & =\{ s \in S : \; \ell( sw) < \ell(w ) \},  \textrm{ the left descent set of $w\in \mathfrak S_n$},
\\
\leq & \textrm{ Bruhat order on $\mathfrak S_n$ (as well as usual order on $\mathbb R$)},
\\
\textrm{$[u,v]$} & =\{ w \in \mathfrak S_n \, : \; u \leq w \leq v \}, \textrm{ the (Bruhat) interval generated by $u,v\in \mathfrak S_n$},
\\
B(\mathfrak S_n)  &  \textrm{ the {\em Bruhat graph  of $\mathfrak S_n$}},
\\
B(X)  &  \textrm{ the (directed) graph induced on the subset $X$ by the Bruhat graph  $B(\mathfrak S_n)$},
\\
P_h(u,v) & \textrm{  the set of paths of length $h$ from $u$ to $v$ in $B(\mathfrak S_n)$}.
\end{array}$
\bigskip

Recall that the \emph{Bruhat graph} $B(\mathfrak S_n)$ is the edge-labelled directed graph with vertex set $\mathfrak S_n$ where  $u \xrightarrow{t} v$ if and only if $vu^{-1}=t \in T$ and $\ell(u)<\ell(v)$. Recall also that the \emph{Bruhat order} is the partial order on $\mathfrak S_n$ where $u \leq v$ provided that  there is a (directed) path from $u$ to $v$ in $B(\mathfrak S_n)$. (All paths that we consider in a directed graph are tacitly assumed to be directed.)

The following result (see 
\cite[Proposition~3.3]{DyeComp1}) implies that the Combinatorial Invariance Conjecture is equivalent to the conjecture that one  obtains by considering the isomorphism type of the graph $B([u,v])$ induced by the Bruhat graph on the elements of the Bruhat interval $[u,v]$.

\begin{thm}
Let $u,v\in \mathfrak S_n$. The isomorphism class (as a poset) of the  Bruhat interval $[u,v]$ determines the isomorphism class (as a directed graph) of the graph $B([u,v])$. 
\end{thm}

Call a directed graph, respectively,   a segment, a diamond, or a $k$-crown,  if it is isomorphic, respectively, to the first, the second or the third directed graph depicted in Figure~\ref{esporoma} (the third graph has $2k+2$ vertices).  We call by the same name a poset whose Hasse diagram is a segment, a diamond, or a $k$-crown.

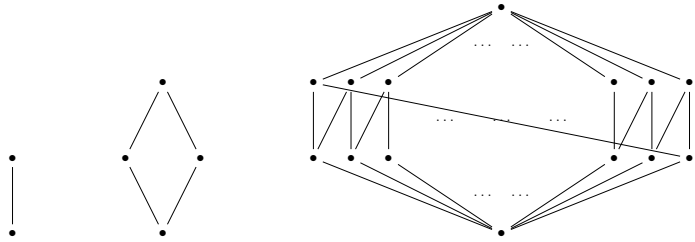
\begin{figure}[h]
    \centering
\scalebox{.5}{
    \begin{tikzpicture}
 \node (p) at (-13,0) {$\bullet$};
 \node (q) at (-13,2) {$\bullet$};

 \node (l) at (-9,0) {$\bullet$};
 \node (m1) at (-10,2) {$\bullet$};
 \node (m2) at (-8,2) {$\bullet$};
 \node (n) at (-9,4) {$\bullet$};

    \node (u) at (0,0) {$\bullet$};
    
    \node (a1) at (-5,2) {$\bullet$};
    \node (a2) at (-4,2) {$\bullet$};
    \node (a3) at (-3,2) {$\bullet$};
    \node (b1) at (-5,4) {$\bullet$};
    \node (b2) at (-4,4) {$\bullet$};
    \node (b3) at (-3,4) {$\bullet$};
    
	\node (a6) at (5,2) {$\bullet$};
    \node (a5) at (4,2) {$\bullet$};
    \node (a4) at (3,2) {$\bullet$};
    \node (b6) at (5,4) {$\bullet$};
    \node (b5) at (4,4) {$\bullet$};
    \node (b4) at (3,4) {$\bullet$};    
    
    \node (v) at (0,6) {$\bullet$};
    
    \node at (-0.5,5) {$\ldots$};
    \node at (0.5,5) {$\ldots$};
    
    \node at (-1.5,3) {$\ldots$};
    \node at (0,3) {$\ldots$};
    \node at (1.5,3) {$\ldots$};

    \node at (-0.5,1) {$\ldots$};
    \node at (0.5,1) {$\ldots$};
    
    \path[-]
(p) edge (q)

(l) edge (m1)
(l) edge (m2)
(n) edge (m1)
(n) edge (m2)

    (u) edge (a1)
    (u) edge (a2)
    (u) edge (a3)
    
    (u) edge (a4)
    (u) edge (a5)
    (u) edge (a6)
    
    (a1) edge (b1)
    (a1) edge (b2)
    (a2) edge (b2)
    (a2) edge (b3)
    (a3) edge (b3)
    
    (a4) edge (b4)
    (a4) edge (b5)
    (a5) edge (b5)
    (a5) edge (b6)
    (a6) edge (b6)
    
    (a6) edge (b1)
    
    (b1) edge (v)
    (b2) edge (v)
    (b3) edge (v)
    
    (b4) edge (v)
    (b5) edge (v)
    (b6) edge (v);    
\end{tikzpicture}
 }
    \caption{A segment, a diamond, and a $k$-crown}
    \label{esporoma}
\end{figure}

\begin{lem}
\label{sangiovanni}
Let $u,v \in \mathfrak S_n$, with $u \leq v$ and $\ell(v) - \ell(u) =3$ ($\ell(v) - \ell(u) =2$, respectively). Then the Hasse diagram of the Bruhat interval $[u,v]$ is a $k$-crown for some $k \in \{2,3,4\}$ (a diamond, respectively).
\end{lem}

The following result-definition is specific to the symmetric group. 
\begin{defn}
Let $u,v \in \mathfrak S_n$.  Given a path  $\Gamma= \big( u {\longrightarrow} x   {\longrightarrow}  v\big)$ of length 2, there exists a unique  path $f(\Gamma)$ of the form $ \big( u {\longrightarrow} y {\longrightarrow}  v\big)$, with $y\neq x$. We call $f$ the {\em flip operator}.
\end{defn}

Next result immediately follows by \cite[Corollary~2.3]{BI}. 
\begin{lem}
\label{B-I}
    Let $\preceq$ be a reflection order. Let $\Gamma$ be a path of length 2 of the form   $\big(\bullet \stackrel{t}{\longrightarrow}  \bullet  \stackrel{r}{\longrightarrow} \bullet \big) $, with $t \prec r$. Let $\big(\bullet \stackrel{p}{\longrightarrow}  \bullet  \stackrel{q}{\longrightarrow} \bullet \big) $ be the form of  the flip of $\Gamma$. Then $ q \prec p$, $ t \prec p$, and $ q \prec r$.
\end{lem}

A  total order $\preceq$ on the set of transpositions $T$ is a \emph{reflection ordering} provided that,  given $a,b,c \in [n]$ with   $a<b<c$,   either
\begin{eqnarray}
\label{ordineriflessione}
 (a,b) \preceq (a,c) \preceq (b,c) \qquad \text{or} \qquad (b,c) \preceq (a,c) \preceq (a,b). 
\end{eqnarray}

The lexicographic order $(1,2) \preceq (1,3) \preceq \cdots \preceq (1,n) \preceq (2,3) \preceq \cdots \preceq (n-1,n)$ and the reverse lexicographic order $(n,n-1) \preceq (n,n-2) \preceq \cdots \preceq (n,1) \preceq (n-1,n-2) \preceq \cdots \preceq (2,1)$ are instances of reflection orderings.

 Let $u,v \in \mathfrak S_n$, $u \leq v$, and  $\Gamma= (u=x_0 \rightarrow x_1 \rightarrow \cdots \rightarrow x_h=v)$ be a path from $u$ to $v$ in $B(\mathfrak S_n)$. We define
  $h$ to be the {\em length} of $\Gamma$, denoted $\ell(\Gamma)$. 
Fix a reflection ordering $\preceq $.  A  path $\Gamma= \big(x_0 \stackrel{t_1}{\longrightarrow} x_1 \stackrel{t_2}{\longrightarrow}  \cdots \stackrel{t_{h-1}}{\longrightarrow} x_{h-1}\stackrel{t_h}{\longrightarrow} x_h\big)$ in $B(\mathfrak S_n)$ is \emph{increasing (with respect to $\preceq $)} if $t_1 \preceq t_2\preceq \cdots \preceq t_h$.

 In \cite{K-L}, Kazhdan and Lusztig introduce two  families of polynomials $\{ P_{u,v} \}_{u,v \in \mathfrak S_n} \subseteq \mathbb Z[q]$ and $\{ R_{u,v} \}_{u,v \in \mathfrak S_n} \subseteq \mathbb Z[q]$, which are now known as the Kazhdan--Lusztig  and $R$-polynomials, respectively, of $\mathfrak S_n$ (actually, the group could be any Coxeter group, see \cite[\S 5]{BB}). Kazhdan--Lusztig polynomials and $R$-polynomials are equivalent: more precisely, given $u,v \in \mathfrak S_n$, it is possible to compute the set $\{ P_{x,y} \}_{x,y \in [u,v]}$ once one knows the set $\{ R_{x,y} \}_{x,y \in [u,v]}$, and vice versa.

We do not define Kazhdan--Lusztig and $R$-polynomials here because we work with another family of polynomials, the 
$\widetilde{R}$-polynomials, which are a rescaling of the $R$-polynomials:  the polynomial $\widetilde{R}_{u,v}(q)$ associated with $u$ and $v$ is the unique polynomial with natural coefficients satisfying 
$$
R_{u,v}(q)= q^{\frac{\ell(u,v)}{2}}\widetilde{R}_{u,v}(q^{\frac{1}{2}}-q^{-\frac{1}{2}}),
$$
(see, e.g., \cite[Proposition 5.3.1]{BB}).
The  following combinatorial interpretation due to Dyer (see \cite{DyeComp} and also \cite[Theorem~5.3.4]{BB}) gives another definition of the $\widetilde{R}$-polynomials. 
\begin{thm}
\label{Dyertilde}
Let  $\preceq $ be a reflection ordering, and  $u,v \in \mathfrak S_n$. Then
$$\widetilde{R}_{u,v}(q)=\sum q^{\ell(\Gamma)},$$
where the sum is over all increasing paths $\Gamma$ from $u$ to $v$.
\end{thm}

\section{Time-support graphs and flipclasses}
\label{definizioni}
In this section, we give the definitions of certain new combinatorial objects.

Given a directed graph  $G$, we denote the vertex set and the edge set of $G$ by $V(G)$ and $E(G)$, respectively. Given $h \in \mathbb N$ and $u,v \in V(G)$, we let $P^G_h(u,v)$ denote  the set of paths of length $h$  in $G$ from $u$ to $v$.
\begin{defn}
Let $G$ be an edge-labelled directed graph, $u,v \in V(G)$, and $h \in \mathbb N$.  Let $F$ be a subset of $P^G_h(u,v)$. 

The {\em support graph of $F$}, denoted $S_F$, is the edge-labelled  directed graph such that
\begin{enumerate}
\item $V(S_F)= \{ a \in V(G) : \text{ there exists a path in $F$ passing through $a$} \}$,
\item $E(S_F)= \{ a \stackrel{t}{\longrightarrow}  b \in E(G): \text{ there exists a path in $F$ containing $a  \stackrel{t}{\longrightarrow}  b$}\}$.
\end{enumerate}

The {\em time-support graph of $F$}, denoted $TS_F$, is the edge-labelled  directed graph such that
\begin{enumerate}
\item $V(TS_F)= \{ (a,i) \in G \times \{0,\ldots ,h\}: \text{ $ \exists \big(u=x_0 {\longrightarrow} x_1 {\longrightarrow} \cdots {\longrightarrow} x_{h}=v\big)$ in $F$ with $x_i = a$} \}$ 
\item $E(TS_F)= \{ (a,i) \stackrel{t}{\longrightarrow}  (b, i+1) : \\\text{ $ \exists \big(u=x_0 {\longrightarrow} x_1 {\longrightarrow} \cdots {\longrightarrow} x_{h}=v\big)$ in $F$ with $x_i = a$, $x_{i+1} = b$ and containing $x_i  \stackrel{t}{\longrightarrow}  x_{i+1}$}\}$.
\end{enumerate}
For $i \in \{0,1, \ldots, h\}$, we say that the path $\big(u=x_0 {\longrightarrow} x_1 {\longrightarrow} \cdots {\longrightarrow} x_{h}=v\big)$ passes through vertex $x_i$ at time $i$.
\end{defn}
As an example, consider the graph $G$ depicted in Figure~\ref{espopancia0}, on the left. Edges point upwards and can have arbitrary labels. Let $F$ be the set of all paths from $u$ to $v$ of length 4. Then the support graph $S_F$ is isomorphic to $G$, while the time-support graph $TS_F$ is depicted on the right.

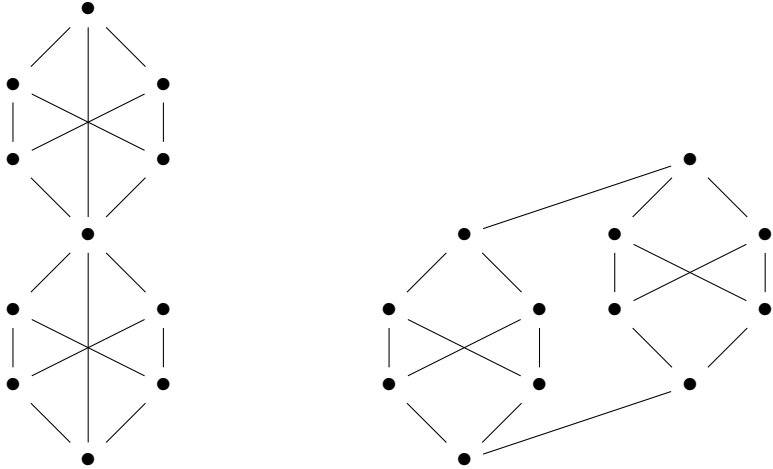
\begin{figure}[h] \label{espomateria}
\centering
    \begin{tikzpicture}
    \node (u) at (-3,0) {$\bullet$};
    
    \node (a1) at (-4,1) {$\bullet$};   
    \node (a2) at (-2,1) {$\bullet$};
       
	\node (b1) at (-4,2) {$\bullet$};
    \node (b2) at (-2,2) {$\bullet$};
    
    \node (c1) at (-3,3) {$\bullet$};
    
    \node (d1) at (-4,4) {$\bullet$};   
    \node (d2) at (-2,4) {$\bullet$};
       
	\node (e1) at (-4,5) {$\bullet$};
    \node (e2) at (-2,5) {$\bullet$};
    
    \node (f1) at (-3,6) {$\bullet$};

    \node (uu) at (2,0) {$\bullet$};
    
    \node (aa1) at (1,1) {$\bullet$};   
    \node (aa2) at (3,1) {$\bullet$};
       
	\node (bb1) at (1,2) {$\bullet$};
    \node (bb2) at (3,2) {$\bullet$};
    
    \node (cc1) at (2,3) {$\bullet$};
    
    \node (cc2) at (5,1) {$\bullet$};
    
    \node (dd1) at (4,2) {$\bullet$};   
    \node (dd2) at (6,2) {$\bullet$};
       
	\node (ee1) at (4,3) {$\bullet$};
    \node (ee2) at (6,3) {$\bullet$};
    
    \node (ff1) at (5,4) {$\bullet$};    
   
    \path[-]
    (u) edge (a1)
    (u) edge (a2)
    
    (a1) edge (b1)
    (a1) edge (b2)
    (a2) edge (b1)
    (a2) edge (b2)
    
    (b1) edge (c1)
    (b2) edge (c1)
    
    (c1) edge (d1)
    (c1) edge (d2)
    
    (d1) edge (e1)
    (d1) edge (e2)
    (d2) edge (e1)
    (d2) edge (e2)
    
    (e1) edge (f1)
    (e2) edge (f1)
    
    (u) edge (c1)
    (c1) edge (f1)

    (uu) edge (aa1)
    (uu) edge (aa2)
    
    (aa1) edge (bb1)
    (aa1) edge (bb2)
    (aa2) edge (bb1)
    (aa2) edge (bb2)
    
    (bb1) edge (cc1)
    (bb2) edge (cc1)
    
    (uu) edge (cc2)

    (cc2) edge (dd1)
    (cc2) edge (dd2)
    
    (dd1) edge (ee1)
    (dd1) edge (ee2)
    (dd2) edge (ee1)
    (dd2) edge (ee2)
    
    (ee1) edge (ff1)
    (ee2) edge (ff1)
    
    (cc1) edge (ff1)

    ;   
\end{tikzpicture}
    \caption{Support and time-support graphs}
    \label{espopancia0}
\end{figure}

\begin{rmk}
\label{osservazione}
\begin{enumerate}
\item
\label{ipo} $S_F$ and $TS_F$ are both acyclic directed graphs, and so both induce a poset structure on the set of vertices. Moreover, $TS_F$ is the Hasse diagram of the graded poset that it induces. 
\item 
\label{oss stesso}
The forgetful map $(x,i) \mapsto x$ from $TS_F$ to $S_F$ is a label preserving surjective map of directed graphs. It is an isomorphism of graphs  if and only it is injective. Equivalently, it is an isomorphism  if and only two paths of $F$ sharing the same vertex $x$ pass through $x$ at the same $i$.
\item  $S_F$ is a subgraph of $G$, while $TS_F$, in general, is not isomorphic to a subgraph of $G$.
\item There are two natural label-preserving injective maps  $i_{S_F} : F \to P^{S_F}_h(u,v)$ and  $i_{TS_F} : F \to P^{TS_F}_h(u,v)$. 
\end{enumerate}
\end{rmk}

\begin{defn}
Let $G$ be an edge-labelled directed graph, $u,v \in V(G)$, and $h \in \mathbb N$.  Let $F$ be a subset of $P^G_h(u,v)$.
\begin{itemize} 
\item The {\em $t$-vector} of $F$, denoted $t_F$,  is $t_F=(t_0, t_1, \ldots, t_{h})$, where $t_i$ is the cardinality of the  set
$$\{x : \text{ there exists a path  $\Gamma = \big(x_0 {\longrightarrow} x_1 {\longrightarrow} \cdots {\longrightarrow} x_{h} \big)$ in $F$ such that $x_i = x$} \}.$$

\item The {\em $\iota$-polynomial} of $F$, denoted  $\iota_F (x,y,t)$,  is 

$$\iota_F (x,y,t) = \sum_{(a,i) \in V(TS_F)} x^{\operatorname{indegree}(a,i)} \; y^{\operatorname{outdegree}(a,i)} \; t^i.$$
\end{itemize}
\end{defn}

\begin{rmk}
Notice that $t_F$ is the rank vector of the poset induced by $TS_F$ (see (\ref{ipo} of Remark~ \ref{osservazione}), i.e. $t_i$ is the cardinality of the  set
$\{(x,i) : (x,i) \in V(TS_F) \}.$
Notice also that the $t$-vector of $F$ is encoded in the $\iota$-polynomial of $F$: indeed, the $t$-vector is the sequence of the coefficients of  $\iota_F (1,1,t)$. Other specializations of the $\iota$-polynomial of $F$ give other pieces of  information on $TS_F$: for example, $\frac{1}{i!}\frac{\partial^{i} \iota}{\partial t^{i}}(x,1,0)$ (respectively, $\frac{1}{i!}\frac{\partial^{i} \iota}{\partial t^{i}}(1,y,0)$) gives the sequences of the in-degrees (respectively, out-degrees) of the vertices of $TS_F$ of time $i$.

\end{rmk}

\begin{defn}
\label{effettivo}
We say that a path $\Delta$ of $S_F$ (of $TS_F$, respectively) is {\em effective} if there exists a path $\Gamma$ in $F$ such that $\Delta$ is a subpath of $i_{S_F}(\Gamma)$ (of  $i_{TS_F}(\Gamma), respectively)$.
\end{defn}
Recall that $B(\mathfrak S_n)$ denotes the Bruhat graph of $\mathfrak S_n$. For short, we denote $P^{B(\mathfrak S_n)}_{h}(u,v)$ simply by $P_{h}(u,v) $.
Recall that, given $u,v \in \mathfrak S_n$, the set  $ P_{2}(u,v) $ has cardinality either 0 or 2. If $ P_{2}(u,v) = \{\Gamma_1, \Gamma_2\} $, we say that $\Gamma_2$ is the flip of $\Gamma_1$, and vice versa. 

\begin{pro}
\label{regoladiflip}
Let $\Gamma_1 = \big( u\stackrel{p}{\longrightarrow}  x\stackrel{q}{\longrightarrow}  v \big)$ be a path of length 2 in 
$B(\mathfrak S_n)$. Let $\Gamma_2 = \big( u\stackrel{s}{\longrightarrow}  y\stackrel{t}{\longrightarrow}  v\big)$ be the flip of $\Gamma_1$.
Then:
\begin{enumerate}
\item if $p$ and $q$ commute, then $s=q$ and $t=p$;
\item if $p = (a, b)$ and $q = (a,c)$:
\begin{itemize}
\item if $a<b<c$ or $a>b>c$, then $s=(b,c)$ and  $t= (a,b)$;
\item if $a<c<b$ or $a>c>b$, then  $s=(a,c)$ and $t= (b, c)$;
\item if $b<a<c$ or $b>a>c$, then:
\begin{itemize}
\item if  $u^{-1}(b) < u^{-1}(a)< u^{-1}(c)$ or $u^{-1}(c) < u^{-1}(a)< u^{-1}(b)$, then  $s=(a,c)$ and $t= (b, c)$;
\item otherwise, $s=(b,c)$ and $t= (a,b)$.
\end{itemize}

\end{itemize}
\end{enumerate}
\end{pro}
\begin{proof}
In all these cases, one checks readily that $ts = qp$, and $u < su < tsu$ in the Bruhat order.
\end{proof}

\begin{defn}
Let $h \in \mathbb N$ and $u,v \in \mathfrak S_n$. 
\begin{itemize}
\item
For $i\in [h-1]$,  the {\em $i$-th flip operator} $f_i: P_{h}(u,v) \to P_{h}(u,v)$ is the map sending a path $\Gamma = \big(u=x_0 {\longrightarrow} x_1 {\longrightarrow} \cdots {\longrightarrow} x_{h}=v\big)$ to the path obtained from $\Gamma$ by replacing the subpath $\big( x_{i-1} {\longrightarrow} x_i {\longrightarrow}  x_{i+1}\big)$ with its flip. 
\item
We call the {\em $h$-flip group of $u$ and $v$}, denoted $\mathfrak F_{h,u,v}$, the subgroup of the symmetric group on  $P_{h}(u,v)$ generated by the set $\{f_i : i \in [h-1] \}$ of all flip operators.
\item
Given $\Gamma \in P_{h}(u,v)$, the {\em flipclass of $\Gamma$} is the orbit of $\Gamma$ under the action of $\mathfrak F_{h,u,v}$.
\item If $F$ is an orbit of $\mathfrak F_{h,u,v}$, we say that $F$ is an $h$-flipclass of $\mathfrak S_n$. 
\end{itemize}
\end{defn}

We are interested in support graphs and time-support graphs associated with flipclasses.
\begin{rmk}
\label{osservazione2}
\begin{enumerate}
\item Each flip operator is a non trivial involution, i.e. it has order 2.
\item If $F$ is a flipclass,  then the  poset induced on $V(S_F)$ by $S_F$  injects in $\mathfrak S_n$ under Bruhat order (however, given two incomparable vertices in $V(S_F)$, they could be comparable in Bruhat order).
\item Let $F$ be a flipclass. If $(a,i) \stackrel{t}{\longrightarrow}  (b,i+1) \in E(TS_F)$,   then $a \stackrel{t}{\longrightarrow} b \in E(S_F)$ (as well as $a \stackrel{t}{\longrightarrow} b \in E(B(\mathfrak S_n))$.
\item Recall that, if $x \longrightarrow y \in E(B(\mathfrak S_n))$, then $\ell(y) - \ell(x) \equiv 1$   $(\mathrm{mod} \; 2)$. Hence, given a flipclass $F$,  if $(a, i),(a,j) \in V(TS_F)$, then $i \equiv j$ $ (\mathrm{mod} \; 2)$.
\end{enumerate}
\end{rmk}

\begin{defn}
Let $n,m \in \mathbb N^+$ and $h \in \mathbb N$. Let $F$ be an $h$-flipclass  of $\mathfrak S_n$ and $F'$ be an $h$-flipclass  of $\mathfrak S_m$. 
We say that $(f,g)$ is an isomorphism from $F$ to $F'$  provided that:
\begin{itemize}
\item $f$ is a bijection from the set of elements of the paths in $F$ to the set of elements of the paths in $F'$ that induces a flips-preserving bijection from $F$ to $F'$;
\item $g$ is a bijection from the set of transpositions labeling the paths in $F$ to the set of transpositions labeling the paths in $F'$;
\item if $(t_1, \ldots t_h)$ is the ordered sequence of the labels of a path $\Gamma$ in $F$, then $(g(t_1), \ldots g(t_h))$ is the ordered sequence of the labels of the path $f(\Gamma)$ of $F'$;
\item $g$ is order preserving w.r.t. the lexicographical orders on $\mathfrak S_n$ and $\mathfrak S_m$. 
\end{itemize}
\end{defn}
Analogously, we may introduce the notion of isomorphism between support graphs and time-support graphs. Clearly, isomorphic flipclasses have isomorphic support graphs,  isomorphic  time-support graphs, and same $\iota$-polynomial.

Let $F$ be a flipclass. We also consider the unlabelled version $F^u$, $S^u_F$, and $TS^u_F$ obtained by forgetting the labels of the edges of $F$, $S_F$, and $TS_F$, respectively, together with the natural notions of isomorphisms.
In particular, given two $h$-flipclasses $F$ and $G$, an isomorphism from $F^u$ to $G^u$ is a bijection from the set of elements of the paths in $F$ to the set of elements of the paths in $G$ that induces a flips-preserving bijection from $F$ to $G$.

\begin{rmk}
\label{costante}
Let $f$ be a function with the set of Bruhat intervals in $\mathfrak S_n$ as domain. Recall from the introduction that we say that $f$ is a {\em combinatorial invariant} provided that $f$ is constant on the isomorphism classes (in the category of posets). 
Given $u,v \in \mathfrak S_n$, with $u\leq v$, and $h\in \mathbb N$, the set of the isomorphism classes of the unlabelled $h$-flipclasses of paths from $u$ to $v$, the set of the isomorphism classes of their unlabelled support graphs, the set of the isomorphism classes of the their unlabelled time-support graphs, and the set of their $\iota$-polynomials  are combinatorial invariants. 
\end{rmk}

\section{Flipclasses and number of increasing paths}

In this section, we prove a first basic result in support of our guiding principle that what happens for intervals should in fact be viewed  as shadows of finer facts for flipclasses. 
Theorem~\ref{Dyertilde} implies that an interval has the same number of increasing paths w.r.t. any reflection ordering. As far as we know, there are no direct proofs of this fact. Here we give an elementary and conceptual proof of the finer fact that a flipclass has the same number of increasing paths w.r.t. any reflection ordering. This sheds new light also on the case of intervals.

\begin{thm}
\label{prec=prec}
A flipclass $F$ has the same number of increasing paths w.r.t. any reflection ordering.
\end{thm}
\begin{proof}
It is well known (see \cite{Dyeth}, Remark 6.16,(iv)) that the reflection orderings of $\mathfrak S_n$ are in bijection with the reduced expressions of the longest permutation:  $t_1 \preceq t_2 \preceq \cdots  \preceq t_{{n}\choose{2}}$  is a reflection ordering if and only if  there is a reduced expression $s_1 \cdots s_{{n}\choose{2}}$  of the longest permutation  such that $t_i = s_1 \cdots s_i  \cdots s_1$.
It is also well known (see \cite{Tit})  that two reduced expressions of the same permutation are linked by a finite sequence of braid moves. A braid move is the replacement of a consecutive subsequence $ss's \cdots$ with $s'ss' \cdots$, where the two sequences have as many factors as the order of the product $ss'$. We call a braid move short if $s$ and $s'$ commute.

Let $\preceq$ and $\preceq'$ be two reflection orderings, and let $\lhd$ and $\lhd'$ denote the covering relations of $\preceq$ and $\preceq'$, respectively. Denote by $t_i$ (respectively, $t_i'$) the $i$-th reflection in the order $\preceq$ (respectively, $\preceq'$).  We find a bijection $\phi$ from the set $I_{\preceq}$ of increasing paths of $F$ w.r.t. $\preceq$ to the  set $I_{\preceq'}$ of increasing paths of $F$ w.r.t. $\preceq'$.  By transitivity, we may suppose that the reduced expressions associated with $\preceq$ and $\preceq'$  differ by a braid move. Let $s_1 \cdots s_n$ be the reduced expression associated with $\preceq$.  
 
First suppose that the reduced expression associated with $\preceq'$ is obtained from this expression by a short braid move, say swapping $s_k$ and $s_{k+1}$ with $s_ks_{k+1}=s_{k+1}s_{k}$  Let $\sigma$ denote $s_1 \cdots s_{k-1}$. Then
$t_i = t_i'$ for $i\notin \{k, k+1\}$, $t_k =\sigma s_{k} \sigma^{-1}$,  $t_k' =\sigma s_{k+1} \sigma^{-1}$, 
$t_{k+1} =\sigma s_{k} s_{k+1} s_k \sigma^{-1} = t_k'$,  $t_{k+1}' =\sigma s_{k+1} s_k s_{k+1} \sigma^{-1}= t_k$, i.e., $\preceq'$ is obtained from $\preceq$ by changing the order of $t_k$ and $t_{k+1}$ ($t_k\lhd t_{k+1}$ while $t_{k+1}\lhd t_{k}$).
Note 
\begin{eqnarray}
\label{tt=tt}
t_kt_{k+1} = t_{k+1}t_k
\end{eqnarray}

Let $\Gamma \in I_{\preceq}$. If $t_k$ and $t_{k+1}$ are not both labels of $\Gamma$, then $\Gamma \in I_{\preceq'}$ and we set $\phi(\Gamma) = \Gamma$.  Suppose $t_k$ and $t_{k+1}$ are both labels of $\Gamma$. Then, since $t_{k}\lhd t_{k+1}$ and $\Gamma$ is increasing w.r.t. $\preceq$, the target of the edge labeled by $t_k$ coincides with the source of the edge labelled by $t_{k+1}$ (and clearly there are no other edges with labels in $\{t_k, t_{k+1}\}$). The flip of $\Gamma$ in correspondence with these two edges gives a path not only in $F$, but also in $I_{\preceq'}$ by Lemma~\ref{B-I}, Proposition~\ref{regoladiflip}, and (\ref{tt=tt}). We set this path be $\phi(\Gamma)$. 

\medskip
Now suppose that the reduced expression associated with $\preceq'$ is obtained from the expression associated with $\preceq$ by a  braid move that is not short, say by replacing $s_ks_{k+1}s_{k+2}$ with $s_{k+2}=s_k$ by $s_{k+1}s_{k}s_{k+1}$. For short, let $s=s_k$ (so $s=s_{k+2}$) and $r=s_{k+1}$.  Let $\sigma$ denote $s_1 \cdots s_{k-1}$. Then $t_i = t_i'$ for $i\notin \{k, k+2\}$,  $t_k =\sigma s \sigma^{-1}$,  $t_k' =\sigma r \sigma^{-1}$, 
$t_{k+2} =\sigma srs r s \sigma^{-1}= \sigma r  \sigma^{-1}$,  $t_{k+2}'  =\sigma rsrsr \sigma^{-1}= \sigma s \sigma^{-1}$,
i.e. $\preceq'$ is obtained from $\preceq$ by swapping $t_k$ and $t_{k+2}$. 
Indeed, $t_{k+1} =\sigma s r s \sigma^{-1}  =\sigma rsr \sigma^{-1}= t_{k+1}'$.
Note 
\begin{eqnarray}
\label{ttt=ttt}
t_kt_{k+1}t_k = t_{k+1}t_kt_{k+1}
\end{eqnarray}

Let $\Gamma \in I_{\preceq}$. If none of $t_k$, $t_{k+1}, t_{k+2}$ are  labels of $\Gamma,$ then $\Gamma \in I_{\preceq'}$ and we set $\phi(\Gamma) = \Gamma$.  If exactly one of $t_k$, $t_{k+1}, t_{k+2}$ is a label of $\Gamma$, then again $\Gamma \in I_{\preceq'}$ and we set $\phi(\Gamma) = \Gamma$.
If exactly two of $t_k$, $t_{k+1}, t_{k+2}$ are  labels of $\Gamma$, then, as above, the two edges must be consecutive in $\Gamma$. The flip of $\Gamma$ in correspondence with these two edges gives a path in $I_{\preceq'}$ (by Lemma~\ref{B-I} and Proposition~\ref{regoladiflip}) that we set to be  $\phi(\Gamma)$.

If exactly three of $t_k$, $t_{k+1}, t_{k+2}$ are  labels of $\Gamma$, then $\Gamma$ has a subpath of the form  $\big( \bullet  \stackrel{t_k}{\longrightarrow} \bullet   \stackrel{t_{k+1}}{\longrightarrow} \bullet  \stackrel{t_{k+2}}{\longrightarrow} \bullet\big)$. By flipping twice  (in any order) in correspondence with the two permutations in the middle (similarly to what happens in $\mathfrak S_3$), we obtain a path in $I_{\preceq'}$ by  Proposition~\ref{regoladiflip} and (\ref{ttt=ttt}).  We set this path to be  $\phi(\Gamma)$.

 The map $\phi$ so defined is an involution by Proposition~\ref{regoladiflip}. 
\end{proof}

After Theorem~\ref{prec=prec}, we can give the following definition.
\begin{defn}
Given a flipclass $F$, we let $c(F)$ be the number of increasing paths in $F$ (w.r.t. any reflection ordering).
\end{defn}

\section{Some results on flipclasses}

In this section, we provide some results on flipclassess, their support graphs, and their time-support graphs that are needed in what follows.

The following result  is basically an adjustment of \cite[Proposition~3.9]{Blanco} and  \cite[Proposition~6.3]{BilBre} to flipclasses. Its proof is therefore omitted.
\begin{pro}
\label{BiB}
Let $u,v \in \mathfrak S_n$ and $h \in \mathbb N$. Let  $F$ be an $h$-flipclass of paths from $u$ to $v$. Then the number $c(F)$ satisfies
$$1 \leq c(F) \leq \frac{|F|}{2^{h-1}}.$$
An increasing path of $F$  w.r.t. a reflection ordering $\preceq$ is the lexicographically-first one.
\end{pro}
\begin{rmk}
\label{colexi}
We notice that also the colexicographically-last  path of $F$ is an increasing path (possibly coinciding with the lexicographically-first one).
\end{rmk}

\begin{lem}
\label{almeno2}
Let  $h \in \mathbb N$ and $u,v \in \mathfrak S_n$. Let  $F$ be an $h$-flipclass of paths from $u$ to $v$.
Each vertex $(x,i)$ of $TS_{F}$ with $i>1$  has in-degree at least 2.
Each vertex $(x,i)$ of $TS_{F}$ with $i < h-1$ has out-degree at least 2.
\end{lem}
\begin{proof}
By upside down symmetry, we need to prove only the first assertion.

By the definition of the vertex set of $TS_{F}$, there exists a path in $F$ passing through $x$ at time $i$. Since $i>1$, we can apply the $(i-1)$-th flip operator to this path and obtain another path passing  through $x$ at time $i$. The edges with target $(x,i)$ in these two paths have different sources. 
\end{proof}

We shall use several times the following property: we give it the name $P_{\bar{x},\bar{y}}$ to easily refer to it.
\begin{lem}
\label{P property}
Let  $h \in \mathbb N$ and $i \in [h-1]$. Let $u,v \in \mathfrak S_n$, with $u < v$. Let  $F$ be an $h$-flipclass of paths from $u$ to $v$. The following property holds.
\begin{enumerate}
\item[($P_{\bar{x},\bar{y}}$)] Given two vertices $\bar{x}=(x,i-1)$ and $\bar{y}=(y,i+1)$ of $TS_{F}$, the number of paths of length 2 from $\bar{x}$ to $\bar{y}$  in $TS_{F}$  is 0, 1 or 2.
\end{enumerate}
\end{lem}
\begin{proof}
By definition, every edge of $TS_{F}$ is effective. Hence a path $\bar{\Delta}$ of length 2 from $\bar{x}$ to $\bar{y}$  in $TS_{F}$ determines a path $\Delta$ of length 2 from $x$ to $y$ in $B(\mathfrak S_n)$. 
Moreover, fix  $i$; if $(a, i)$ and $(b,i)$ are distinct vertices of $TS_{F}$, then $a$ and $b$ are distinct element of $\mathfrak S_n$, so distinct paths $\bar{\Delta}$ and $\bar{\Delta'}$ determine distinct paths $\Delta$ and $\Delta '$.
Since  the number of paths of length 2  from $x$ to $y$ in $B(\mathfrak S_n)$ is either 0 or 2, we have the assertion. 
\end{proof}
\begin{rmk}
We notice explicitely that we cannot exclude that there is only one path of length 2 from $\bar{x}=(x,i-1)$ to $\bar{y}=(y,i+1)$  in $TS_{F}$. Indeed, given such a path $\bar{\Delta}$, the two edges of it could (a priori) come from two different paths in $F$, so $\bar{\Delta}$ could be non-effective, and there could be no path in $F$ we can apply the $i$-th flip to in order to obtain a second path of length 2 from $\bar{x}$ to $\bar{y}$  in $TS_{F}$.
\end{rmk}

\begin{lem}
\label{equivariante}
Let  $h \in \mathbb N$. Let $u,v \in \mathfrak S_n$, with $u < v$. Let  $F$ be an $h$-flipclass of paths from $u$ to $v$. Suppose that, for all $i \in [h-1]$ and all $(a,i-1), (b,i+1) \in V(TS_{F})$, the interval $[(a,i-1),(b,i+1)]$ in the poset induced by $TS_{F}$, if nonempty, is a diamond. Then, for $i \in [h-1]$,  there is a natural concept of  i-th flip operator also in $P_{h}^{TS_{F}}((u,0),(v,h))$),  and the map $i_{TS_{F}}$ is equivariant with respect to the action of the flip operators.
\end{lem}
\begin{proof}
Let $p = \big((u=a_0,0), (a_1, 1), \ldots, (a_{h-1}, h-1), (v=a_h, h) \big)\in P_{h}^{TS_{F}}((u,0),(v,h))$. Let $i\in [h-1]$. Then, the interval $[(a_{i-1},i-1),(a_{i+1},i+1)]$ in  the poset induced by $TS_{F}$ is nonempty. Hence, it is a diamond. The two middle elements are $(a_i , i)$ and another element $(a'_i , i)$. The $i$-th flip of $p$ is defined to be the path 
$p' = 
\big( (u,0), \ldots, 
(a_{i-1}, i-1), (a'_{i}, i), (a_{i+1}, i+1),
\ldots, (v, h) \big)$.
Suppose now that $p$ comes from a path 
$\tilde{p}\in F$, i.e., 
$p = i_{TS_{F}}(\tilde{p})$. Then, clearly one has 
$\tilde{p} = (u, a_1, \dots, a_{h-1}, v)$ and its $i$-th flip is 
$\tilde{p}' = (u, \dots, a_{i-1}, a'_i, a_{i+1}, \ldots, v)$.
Thus one has $p' = i_{TS_{F}}(\tilde{p}')$.
\end{proof}
\begin{lem}
\label{tutti effettivi}
Let  $h \in \mathbb N$. Let $u,v \in \mathfrak S_n$, with $u < v$. Let  $F$ be an $h$-flipclass of paths from $u$ to $v$. TFAE.
\begin{enumerate}
\item
\label{1} The map $i_{TS_{F}}$ is bijective, i.e. every path in $TS_{F}$ is effective.
\item
\label{2} Given $(a,i), (b,i+2) \in V(TS_{F})$ with $(a,i) \leq (b,i+2)$, the interval $[(a,i),(b,i+2)]$ is a diamond. Moreover, the flip operators act transitively on  $P_{h}^{TS_{F}}((u,0),(v,h))$.
\end{enumerate}
\end{lem}
\begin{proof}
Suppose (\ref{1}). Let  $(a,i), (b,i+2) \in V(TS_{F})$  with $(a,i) \leq (b,i+2)$, and consider a path $\bar{\Delta}$ from $(u,0)$ to $(v,h)$ passing through both $(a,i)$ and $(b,i+2)$. By hypothesis, there exists $\Delta$ in $F$ such that $i_{TS_{F}}(\Delta)= \bar{\Delta}$. The $(i+1)$-th flip operator sends $\Delta$ to a path $\Delta'$ that also  passes through $a$ and $b$ at time $i$ and $i+2$, respectively. This implies that the path $i_{TS_{F}}(\Delta')$ passes  through both $(a,i)$ and $(b,i+2)$. Hence,  by Lemma~\ref{P property}, the interval $[(a,i),(b,i+2)]$ is a diamond. Thus, by Lemma~\ref{equivariante}, the map $i_{TS_{F}}$ is equivariant with respect to the action of the flip operators. Hence, transitivity of the flip group on   $F$ implies  transitivity of the flip group on  $P_{h}^{TS_{F}}((u,0),(v,h))$. So, (\ref{2}) holds. 

Vice versa, suppose (\ref{2}).  Recall that $i_{TS_{F}}$ is always injective.  By Lemma~\ref{equivariante}, the map $i_{TS_{F}}$ is equivariant with respect to the action of the flip operators. Thus, transitivity of the flip group on  $P_{h}^{TS_{F}}((u,0),(v,h))$ implies  surjectivity of the map $i_{TS_{F}}$. Hence, (\ref{1}) holds. 
\end{proof}

The following result treats the case of $h$-flipclasses of intervals of length $h$.
\begin{lem}
\label{intervallo}
Let  $h \in \mathbb N$. Let $u,v \in \mathfrak S_n$, with $u < v$ and $\ell(v) - \ell(u)=h$. 
Then 
\begin{enumerate}
\item
\label{111} 
there is only one $h$-flipclass $F$ of paths from $u$ to $v$;
\item
\label{222}
 $S_F$ and $TS_F$ are both isomorphic to the Hasse diagram of the Bruhat interval $[u,v]$;
\item 
\label{333} all paths of $S_F$ and $TS_F$ are effective;
\item 
\label{444}
c(F)=1;
\item 
\label{555}the $t$-vector $t_F= (t_0, \ldots, t_h)$ satisfies $\sum (-1)^i t_i =0$.
\end{enumerate}
\end{lem}
\begin{proof}
A path  from $u$ to $v$ of length $h$ corresponds to a maximal chain in the Bruhat interval $[u,v]$. The fact that the order complex of a Bruhat interval is shellable (see \cite{BW}) implies that every two paths from $u$ to $v$ of length $h$  are flip-equivalent, hence (\ref{111}) holds.  Since $\ell(v) - \ell(u)=h$, every edge of $S_F$ (and of $TS_F$) corresponds to a cover relation of Bruhat order since an edge of the form $x \longrightarrow y$, with $\ell(y) - \ell(x)>1$, cannot be part of a maximal chain. Vice versa, every cover relation belongs to a maximal chain. Hence (\ref{222}) holds. Property (\ref{333}) follows by  (\ref{111})  and(\ref{222}). Property   (\ref{444}) is implied by \cite[Proposition~4.3]{DyeComp}. Property  (\ref{555}) follows by the fact that Bruhat intervals are Eulerian (see \cite[Corollary~2.7.11]{BB}).
\end{proof}

\section{From infinite to finite}

Fix $h$ in $\mathbb N$. 
Consider the $h$-flipclasses of any $\mathfrak S_n$
(letting $n$ vary).
These are clearly infinite in number. The aim of this section is to  show that, on the other hand, the number of isomorphism classes of (labelled) $h$-flipclasses is finite and that all such isomorphism classes can be constructed from flipclasses of $\mathfrak S_{h+1}$.

Let $n \in \mathbb N^+$ and $u,v \in \mathfrak S_n$, with $u < v$. Let $\Gamma= \big(u=x_0 \stackrel{t_1}{\longrightarrow} x_1 \stackrel{t_2}{\longrightarrow}  \cdots \stackrel{t_{h-1}}{\longrightarrow} x_{h-1}\stackrel{t_h}{\longrightarrow} x_h=v\big)$  be a path from $u$ to $v$ in $B(\mathfrak S_n)$ of length $h$. We denote by $W(\Gamma)$ the reflection subgroup of $ \mathfrak S_n$ generated by the transpositions $t_1, \ldots, t_h$. Moreover, we denote by $G(\Gamma)$ the edge-labeled undirected graph such that:
\begin{enumerate}
\item $\{a : \text{there exists $i$ in $[h]$  such that $t_i(a)\neq a$} \}$ is its vertex set;
\item there is an edge between $a$ and $b$ with label $i$ if $t_i = (a,b)$.
\end{enumerate}
Notice that $G(\Gamma)$ may have multiple edges.

\begin{lem}
Let $\Gamma$ and $\Gamma '$ be two paths  in $B(\mathfrak S_n)$  in the same flipclass. Then:  
\begin{itemize}
\item the groups $W(\Gamma)$ and  $W(\Gamma')$ coincide;
\item the vertex sets of $G(\Gamma)$ and  $G(\Gamma')$ coincide;
\item the number of connected components of $G(\Gamma)$ and $G(\Gamma')$ coincide.
\end{itemize}
\end{lem}
\begin{proof}
By transitivity, it is enough to prove the statements when $\Gamma'$ is a flip of $\Gamma$, i.e.,  $\Gamma' = f_i(\Gamma)$ for some $i$  in $[h-1]$. In this case, the assertions readily follow by  Proposition~\ref{regoladiflip}.
\end{proof}
\begin{defn}
Let $F$ be a flipclass.
\begin{itemize}
\item We denote by $W(F)$ the reflection subgroup generated by the  transpositions labeling the edges of $\Gamma$, for any $\Gamma \in F$.
\item We denote by $E(F)$ the vertex set of $G(\Gamma)$, for any $\Gamma$ in $F$.  
\item We say that a flipclass $F$ is reducible if $W(F)$ is a nontrivial direct product, i.e., $G(\Gamma)$ is not connected,  for any $\Gamma \in F$. Otherwise, we say that $F$ is irreducible.
\end{itemize}
\end{defn}
\begin{rmk}
\begin{enumerate}
\item The group $W(F)$ is a direct product of symmetric groups.
\item The partition of $E(F)$ given by the connected components of $G(\Gamma)$ does not depend  on the choice of $\Gamma$ in $F$. 
\end{enumerate}
\end{rmk}

Recall that, given a subset $X$ of $\mathfrak S_n$, we denote by $B(X)$ the (directed) graph induced on  $X$ by the Bruhat graph  $B(\mathfrak S_n)$.
\begin{lem}
\label{FeF'}
Let $W$ be a reflection subgroup of $\mathfrak S_n$. Let $F$ be a flipclass of $\mathfrak S_n$ for which $W(F) = W$. Then\begin{itemize}
\item each path in $F$ is contained in a left coset of $W$,
\item there is a flips and edge labeling preserving bijection between $F$ and a flipclass $F'$ of paths in $B(W)$.
\end{itemize}
\end{lem}
\begin{proof}
The first assertion is straightforward. Let $C$ denote this left coset.

By \cite[Theorem 1.4]{DyeComp1}, there is an edge labeling preserving isomorphism  between the graph $B(C)$ and the graph $B(W)$. This isomorphism clearly preserves the flips. Hence, the second assertion follows.
\end{proof}

The cartesian product of two edge-labelled directed graphs $G_1$ and $G_2$ is the edge-labelled directed graph $G_1 \times G_2$ satisfying $V(G_1 \times G_2)= V(G_1) \times V(G_2)$ and $E(G_1 \times G_2) = \{ (x_1, x_2) \stackrel{t}{\longrightarrow} (y_1,y_2) : \text{ $x_1 = y_1$ and $x_2\stackrel{t}{\longrightarrow}  y_2 \in E(G_2)$, or  $x_1 \stackrel{t}{\longrightarrow}  y_1 \in E(G_1)$ and $x_2 = y_2$}\}$. Given a set of paths $F_i$ in $G_i$, for $i\in \{1,2\}$, we denote the set of paths in $G_1 \times G_2$ obtained by all possible shuffles of paths of $F_1$ and $F_2$ by $F_1 \ast F_2$.

\begin{lem}
\label{prodotto diretto 1}
Let $W$ be a reflection subgroup of $\mathfrak S_n$. Suppose $W = W_1 \times W_2$ (direct product). Then the following hold.
\begin{enumerate}
\item
\label{primo}
 $B(W) = B(W_1) \times B(W_2)$.
\item
\label{secondo}
 The set of paths of $B(W)$ is the set of shuffles of the sets of paths of $B(W_1)$ and $B(W_2)$.
\item 
\label{terzo}
Let $\Gamma$ be a shuffle of $\Gamma_1$ and $\Gamma_2$: 
\begin{enumerate}
\item 
\label{a} $G(\Gamma)$ (as unlabelled graph) is the disjoint union of $G(\Gamma_1)$ and $G(\Gamma_2)$ (as unlabelled graphs) and the labels are determined by the shuffle;
\item
\label{b}  any flip of $\Gamma$  corresponds to either a flip of $\Gamma_1$, or a flip of $\Gamma_2$, or a change of a  position in the shuffle.
\end{enumerate}
\end{enumerate}
\end{lem}
\begin{proof}
The proof is easy and left to the reader.
\end{proof}
Given a path $\Gamma$ in $B(W_1) \times B(W_2)$,  denote by $\Gamma_i$ the projection of $\Gamma$ on $B(W_i)$, for $i\in \{1,2\}$.  
\begin{lem}
\label{prodotto diretto 2}
Let $W$ be a reflection subgroup of $\mathfrak S_n$. Suppose $W = W_1 \times W_2$ (direct product). Let $u=(u_1,u_2), v=(v_1,v_2) \in W$, and  $F$ be a flipclass of paths from $u$ to $v$. Let $F_i = \{\Gamma_i : \Gamma \in F\}$, for $i\in \{1,2\}$.
Then the following hold.
\begin{enumerate}
\item
\label{primo'}
 $F_i$ is a flipclass in $B(W_i)$,  for $i\in \{1,2\}$;
\item
\label{secondo'}
 $F = F_1 \ast F_2$; 
\item 
\label{terzo'}
$S_F=S_{F_1} \times S_{F_2}$;
\item
\label{quarto'}
$TS_F=TS_{F_1} \times TS_{F_2}$;
\item
\label{quinto'}
$\iota_F = \iota_{F_1} \cdot \iota_{F_2}$;
\item
\label{sesto'}
$c(F)= c(F_1) \cdot c(F_2)$ (and analogously for the increasing paths of  of $S_F$ and  $TS_F$ w.r.t. any reflection ordering).
\end{enumerate}
\end{lem}
\begin{proof}
(\ref{primo'}) follows by Lemma~\ref{prodotto diretto 1}, (\ref{b}).

(\ref{secondo'}) follows by Lemma~\ref{prodotto diretto 1}, (\ref{secondo}) and Lemma~\ref{prodotto diretto 1},(\ref{b}).

(\ref{terzo'}) and (\ref{quarto'}) follows by (\ref{secondo'}) and the definitions of support graph and time-support graph.

(\ref{quinto'}) follows by (\ref{quarto'}).

(\ref{sesto'}) follows by (\ref{secondo'}), (\ref{terzo'}), and (\ref{quarto'}).
\end{proof}

\begin{defn}
Let $E$ be a subset of $[n]$ of cardinality $m$. Denote by $r_E$ the unique order preserving bijection from $E$ to $[m]$, as well as the following maps that it induces.
\begin{enumerate}
\item Given $w \in \mathfrak S_n$, we denote by $r_E(w)$ the permutation in $\mathfrak S_m$ whose one-line notation is obtained from the one-line notation of $w$ by deleting the numbers not in $E$ and applying $r_E$ to the numbers in $E$.
\item  Let $\Gamma$ be a path in $B(\mathfrak S_n)$ with labels of the type $(a,b)$ with $a,b \in E$. We denote by $r_E(\Gamma)$ the path in $B(\mathfrak S_m)$ obtained from $\Gamma$ by applying $r_E$ to its nodes and  labels.
\item Given a total order on the set of transpositions in $\mathfrak S_n$, we denote by $\preceq_{r_E}$ the total order on the set of transpositions in $\mathfrak S_m$ defined by  $(a,b) \preceq_{r_E} (c,d)$ if and only if $\big( r_{E}^{-1}(a),r_{E}^{-1}(b)\big) \preceq \big( r_{E}^{-1}(c),r_{E}^{-1}(d)\big) $. 
\end{enumerate}
\end{defn}

\begin{lem}
\label{induceordine}
Let  $\preceq$ be a reflection ordering  on $\mathfrak S_n$ and $E$ be a subset of $ [n]$ of cardinality $m$.  The total order $\preceq_{r_E}$  on the set of the reflections of $\mathfrak S_m$  is a reflection ordering. Moreover, if $\preceq$ is the lexicographical order for $\mathfrak S_n$, then  $\preceq_{r_E}$ is the lexicographical order for $\mathfrak S_m$.
\end{lem}
\begin{proof}
Since $r_E$ is order preserving, condition (\ref{ordineriflessione}) of the definition of a reflection order is satisfied by  $\preceq_{r_E}$ if  it is satisfied by  $\preceq$. The second statement is straightforward.
\end{proof}

\begin{lem}
\label{riduzione a E}
Let  $F$ be a flipclass.  Denote $E(F)$ simply by $E$. Let $|E|=m$. The following hold.
\begin{enumerate}
\item
\label{I}
The set $F'= \{r_E(\Gamma) : \Gamma \in F\}$ is a flipclass of $\mathfrak S_m$. The map $r_E$  is a  flips preserving bijection from  $F$  to  $F'$ satisfying  $r_E(l_i(\Gamma)) = l_i (r_E(\Gamma))$ for all $\Gamma \in F$ (where $l_i(\Delta)$ denote the $i$-th label of the path $\Delta$). 
\item 
\label{II}
The map $r_E$ induces an isomorphism of edge-labelled directed graphs  from $S_{F}$ to $S_{F'}$ and an isomorphism of edge-labelled directed graphs  from $TS_{F}$ to $TS_{F'}$. The labels also are matched by $r_E$.
\item
\label{III}
Fix any reflection ordering  $\preceq$ for $\mathfrak S_n$. The increasing paths of $F$,  $S_{F}$, and $TS_{F}$ (w.r.t. $\preceq $) correspond to the increasing paths of $F'$,  $S_{F'}$, and $TS_{F'}$ (w.r.t. $\preceq_{r_E}$), respectively.
\item
\label{IV} The map $r_E$ induces an isomorphism from $F$ to $F'$.
\end{enumerate} 
\end{lem}
\begin{proof}
(\ref{I}). The one-line notations of the permutations in a path in $F$ have each number of $[n]\setminus E$ in the same position. By Proposition~\ref{regoladiflip}, the map $r_E$ is flips preserving and hence $F'$ is a flipclass. It is now straightforward that $r_E$ is a bijection and  satisfies  $r_E(l_i(\Gamma)) = l_i (r_E(\Gamma))$ for all $\Gamma \in F$.

(\ref{II}), (\ref{III}). Straightforward.

(\ref{IV}).
Follows by (\ref{I}), (\ref{III}), and Lemma~\ref{induceordine}.
\end{proof}

\begin{pro}
\label{flipclass-riducibili}
Let  $h \in \mathbb N$. Let $F$ be an $h$-flipclass of $\mathfrak S_n$. Let $E_1, \ldots, E_s$ be the vertices of the connected components of $G(\Gamma)$, for any $\Gamma$ in $F$, and let $e_i$ be the cardinality of $E_i$, for $i\in [s]$. Then the following hold.
\begin{enumerate}
\item $W(F) = \mathfrak S_{E_1} \times  \cdots \times \mathfrak S_{E_s}$;
\label{uno}
\item
\label{due}
 there exist (unique up to isomorphism and order) irreducible flipclasses $F_i$ of $\mathfrak S_{e_i}$, for $i\in [s]$, such that 
$$F \cong F_1 \ast \cdots \ast F_s.$$ 
\item
\label{tre}
Denoting $h_i$ the length of the paths in $F_i$,  for $i\in [s]$, we have $\sum h_i = h$.
\end{enumerate}
\end{pro}
\begin{proof}
(\ref{uno}). If a subgroup of $\mathfrak S_n$ contains the transpositions $(i,j)$ and $(j,k)$, then it contains also $(i,k)$ since $(i,k)= (i,j)(j,k)(i,j)$. Hence, for all $r\in [s]$,  the subgroup $W(F)$ contains all transpositions $(i,j)$ with $i,j \in E_r$, so $W(F)$ contains $\mathfrak S_{E_r}$. The assertion follows.

(\ref{due}). By Lemma~\ref{riduzione a E}, the flipclass $F$ is isomorphic to a flipclass in $\mathfrak S_{\sum e_i}$. The assertion follows by Lemma~\ref{prodotto diretto 2}, 

(\ref{tre}). This is clear.
\end{proof}

\begin{rmk}
The isomorphism of Proposition~\ref{flipclass-riducibili} can be described explicitely.
\end{rmk}
\begin{thm}
\label{h+1}
Let  $h \in \mathbb N$. Let $F$ be an $h$-flipclass of $\mathfrak S_n$, and $e=|E(F)|$. Then there exists an $h$-flipclass $F'$ of $\mathfrak S_e$ that is isomorphic to $F$. In particular, if $F$ is irreducible, such $F'$ may be found in $ \mathfrak S_{h+1}$.
\end{thm}
\begin{proof}
Denote $E(F)$ simply by $E$. The group $W(F)$ is a subgroup of the subgroup $\mathfrak S_{E}$ of $\mathfrak S_n$. By Lemma~\ref{FeF'}, the flipclass $F$ is isomorphic to some flipclass $F'$ in the graph $B(W(F))$, which is a subgraph of $B(\mathfrak S_E)$. By applying the map $r_E$, we get  the first assertion.

Suppose that $F$ is irreducible. Then $e\leq h+1$ since a connected graph with $h$ edges has at most $h+1$ vertices, and the second assertion follows. 
\end{proof}

The following now readily follows.
\begin{cor}
\label{finiti}
Fix $h$ in $\mathbb N$. There are finitely many isomorphism classes of $h$-flipclasses. 
\end{cor}
\begin{proof}
For any $h$-flipclass $F$, we have the obvious bound  $|E(F)|\leq 2h$. Hence, the result follows by Theorem~\ref{h+1}.
\end{proof}

\section{Combinatorial invariance of the coefficients of $\widetilde{R}$-polynomials}
In this section, we propose a novel approach to tackle the problem of the Combinatorial Invariance of the coefficients of the $\widetilde{R}$-polynomials. This strategy allows us to 
prove the combinatorial invariance of the coefficient  of $q^h$ of the $\widetilde{R}$-polynomial $\widetilde{R}_{u,v}(q)$, for $h\leq 6$. For short, we denote this coefficient by $[q^h]\widetilde{R}_{u,v}(q)$. 

This is achieved by showing that, given $h \leq 6$ and $u,v \in \mathfrak S_n$, with $u\leq v$, the multiset of time-support graphs of the $h$-flipclasses of paths from $u$ to $v$ determines $[q^h]\widetilde{R}_{u,v}(q)$.  More precisely, the multiset of $\iota$-polynomials  determines the coefficient $[q^h]\widetilde{R}_{u,v}(q)$, and in fact, for $h\leq 4$, already the multiset of $t$-vectors suffices.
Indeed, being  the multiset of $\iota$-polynomials of the $h$-flipclasses of paths from $u$ to $v$ a combinatorial invariant (see Remark~\ref{costante}), we get the combinatorial invariance of $[q^h]\widetilde{R}_{u,v}(q)$, for $h\leq 6$.

We note that, in theory, we provide explicit formulas for the coefficients $[q^h]\widetilde{R}_{u,v}(q)$, for $h\leq 6$, for any interval $[u,v]$ of any symmetric group $\mathfrak S_n$. 
\bigskip

We now describe our strategy.

Let $u,v \in \mathfrak S_n$. By Theorem~\ref{Dyertilde}, 
the coefficient $[q^h]\widetilde{R}_{u,v}(q)$ is the number of paths  from $u$ to $v$ of length $h$ that are increasing w.r.t. to any reflection ordering. Hence, we have 
$$[q^h]\widetilde{R}_{u,v}(q)= \sum_F c(F),$$
where $F$ ranges over all $h$-flipclasses of paths from $u$ to $v$. Proving that $c(F_1)=c(F_2)$ holds for all $h$-flipclasses $F_1$ and $F_2$ such that $F_1^u$ and $F_2^u$ are isomorphic shows the Combinatorial Invariance of the coefficient of $q^h$ of the $\widetilde{R}$-polynomial.

The following result provides a general possible way to show the desired equality above. For $r\in \mathbb N$, let 
$$\begin{array}{lll}
\mathcal F_r &=& \{ \text{$F$} : \text{$F$ is an $r$-flipclass in some $\mathfrak S_n$}\}\\
\mathcal F_{r,r+1} &=& \{ \text{$F$} : \text{$F$ is an $r$-flipclass in  $\mathfrak S_{r+1}$}\}\\
\end{array}$$

\begin{defn}
Let $h\in \mathbb N$. A function  $INV$ with domain $\mathcal F_h$ is a {\em$h$-combinatorial invariant} provided that $INV(F_1) = INV(F_2)$ for all $h$-flipclasses $F_1$ and $F_2$ such that $F_1^u$ and $F_2^u$ are isomorphic. Moreover, we say that a combinatorial invariant $INV$ is {\em $h$-good} provided that it detects the function $c$, i.e., $c(F_1) = c(F_2)$ for all $h$-flipclasses $F_1$ and $F_2$ satifying $INV(F_1) = INV(F_2)$.   
\end{defn}

\begin{thm}
\label{strategia}
Fix $h$ in $\mathbb N$. Let $INV$ be an $r$-combinatorial invariant for each $r$ in $[h]$, and suppose that 
\begin{itemize}
\item $INV(F\ast G)$ can be obtained as  a function of $INV(F)$ and $INV(G)$, say $INV(F\ast G) = \mu(INV(F),INV(G))$.
\end{itemize}

Let
\begin{eqnarray*}
\mathcal{A}c_r &=& \{ (INV(F),c(F)) : F \in \mathcal F_{r} \}\\
\mathcal{I}c_r &=&\{ (INV(F),c(F)) : F \in \mathcal F_{r,r+1} \}.\\
\mathcal{R}c_r &=& \{ \Big(\mu(i_1,i_2) , c_1 c_2 \Big) : (i_1,c_1) \in \mathcal{A}c_j, (i_2,c_2) \in \mathcal{A}c_k,  \text{ and } j+k=r \}.
\end{eqnarray*}
Then, $\mathcal{A}c_r$ is finite, and 
$$\mathcal{A}c_r = \mathcal{I}c_r \cup \mathcal{R}c_r,$$
for each $r$ in $[h]$.

Hence, letting
\begin{eqnarray*}
\mathcal I_h &=& \{ INV(F) : F \in \mathcal F_{h,h+1} \}.\\
\mathcal R_h &=& \{ \mu(i_1,i_2)  : (i_1,c_1) \in \mathcal{A}c_j, (i_2,c_2) \in \mathcal{A}c_k,  \text{ and } j+k=h \},
\end{eqnarray*}
if the cardinality of $\mathcal{I}c_h \cup \mathcal{R}c_h$ is equal to the cardinality of $\mathcal{I}_h \cup \mathcal R_h$, 
then $INV$ is an $r$-good invariant, for each $r\in [h]$.
\end{thm}
\begin{proof}
Let $r\in [h]$. Let $F\in \mathcal F_r$. 
Suppose that $F$ is irreducible. By Theorem~\ref{h+1}, there exists $G\in \mathcal F_{r,r+1}$ that is isomorphic to $F$: in particular $INV(F)=INV(G)$ and $c(F)= c(G)$. So $(INV(F), c(F)) \in \mathcal{I}c_r$.  

Suppose that $F$ is reducible. Then, by Proposition~\ref{flipclass-riducibili}, there exist $j,k\in \mathbb N^+$ with $j+k=h$, $G\in \mathcal F_{j}$, and $G'\in \mathcal F_{k}$ such that $F$ is isomorphic to $G \ast G'$. In particular, $INV(F)=INV(G\ast G')=\mu(INV(G),INV(G'))$. Clearly, $(INV(G), c(G)) \in \mathcal{A}c_j$ and $(INV(G'), c(G')) \in \mathcal{A}c_k$. Moreover, by Lemma~\ref{prodotto diretto 2}(\ref{sesto'}), we have $c(F)= c(G) c(G')$. So $(INV(F), c(F)) \in \mathcal{R}c_r$.  By induction, we can show that $\mathcal{A}c_r$ is finite.

The second statement follows from the first since being $r$-good is equivalent to requiring that the projection of $\mathcal{A}c_r$ to the first factor is injective.
\end{proof}

The strategy of Theorem~\ref{strategia} is viable for all $h$ in $\mathbb N$, but the complexity of the computations increases very rapidly: we have carried out the prescribed calculation by computer up to $h=6$. 

In fact, for $h\leq 5$,  we present here also a different approach that allows us to  prove more than what follows by the strategy. 
Given an  $h$-flipclass $F$, we determine  an explicit formula for $c(F)$  that depends only on the $\iota$-polynomial. Actually, for $h \leq 4$,  the formula for $c(F)$ depends only on the $t$-vector of $F$.

\subsection{$h=3$} 
The first unknown case is for $h=3$ since $[q^0]\widetilde{R}_{u,v}(q)$ is 1 if $u=v$ and 0 if $u\neq v$,  $[q^1]\widetilde{R}_{u,v}(q)$ is 1 if $u \rightarrow v \in E(B(\mathfrak S_n))$ and 0 otherwise, $[q^2]\widetilde{R}_{u,v}(q)$ is 1 if there is a path from $u$ to  $v$ in $B(\mathfrak S_n)$ of length 2 (and, in fact, there are two such paths) and 0 otherwise.
\begin{thm}
\label{corone}
Let $u,v \in \mathfrak S_n$, with $u < v$.
The group  $\mathfrak F_{3,u,v}$ is a dihedral group. Moreover, given  a 3-flipclass $F$,  the following hold.
\begin{enumerate}
\item The unlabelled support graph $S^u_{F}$  and the unlabelled time-support graph $TS^u_{F}$ are both isomorphic to a $k$-crown, for some $k\in \{2,3,4,5\}$.
\item $k$ coincides with the order of the product $f_2f_1$ on $F$. 
\item The maps $i_{S_{F}}$ and $i_{TS_{F}}$ are bijections (i.e, all paths in $S_{F}$ and $TS_{F}$ are effective) and send the increasing paths to the  increasing paths (w.r.t. any reflection ordering).
\item
\label{44444}
The number $c(F)$ is equal to 1 when $k\in\{2,3,4\}$, to 2 when $k=5$.
\end{enumerate}
\end{thm}
\begin{proof}
 Since $\mathfrak F_{3,u,v}$ is generated by two involutions, it is a dihedral group. 

Let $\Gamma = \big(u {\longrightarrow} x_1   {\longrightarrow} x_2 {\longrightarrow} v\big)$ be a path in $F$.  Let us construct $F$ from $\Gamma$ by alternately applying the flips $f_1$ and $f_2$ (i.e., we consider $\Gamma$, $ f_1(\Gamma)$,  $ f_2f_1 (\Gamma)$,  $ f_1f_2f_1 (\Gamma)$, and so on). When in this process, the action of $f_1$ produces a new element, then also the subsequent action of $f_2$ produces a new element, since there are at most two paths of length 2 with the same starting element and ending element. When, in this process, the action of $f_1$ does not produce a new element, and so we necessarily come back to $x_1$ (again since there are at most two paths of length 2 with the same starting element and ending element), then with the subsequent action of $f_2$ we get back to $\Gamma$. This implies that $S_F$ is a  $k$-crown, where  $k$ is the order of the product $f_2f_1$ on $F$. Since two paths sharing the same vertex $x$ pass through $x$ at the same time,  $S_F^u$  and  $TS_F^u$ are isomorphic.

Moreover, the paths corresponding to the paths of $F$ exhaust all paths in the $k$-crown, and increasing paths correspond to increasing paths.

It remains to prove $k\in\{2,3,4,5\}$ and (\ref{44444}).


If $F$ is reducible, then $F=F_1\ast F_2$ for some $1$-flipclass $F_1$ and some $2$-flipclass $F_2$, and $S_F^u$ and $TS_F^u$ are both isomorphic to a 3-crown. By Lemma~\ref{prodotto diretto 2}(\ref{sesto'}), we have  $c(F)=1$  since $c(F_1)=1$ and $c(F_2)=1$.

Suppose $F$ is irreducible. By Theorem~\ref{h+1}, there exists a $3$-flipclass $G$ of $\mathfrak S_4$ that is isomorphic to $F$.  

Recall that, for every element $x$ in a Coxeter group, the number of reflections $t$ such that $xt\leq x$ is equal to the length of $x$. Hence, as a general fact, we note that, if  $S_{F}^u$ is a $k$-crown, then $\ell(v)\geq k$ and  $\ell(u)\leq |T| - k$. Since 6 is the length of the longest element $w_0$ of $\mathfrak S_4$, which is $4321$, the graph $S_{F}^u$ is a $k$-crown with $k\leq 6$. Actually, $k$ cannot be 6 because the only pair of elements $(u,v)$ with $\ell(v) \geq 6$ and $\ell(u)\leq |T| - 6 =0$ is $(e,w_0)$ and the interval between this two elements has length 6, thus even, and hence has no paths of length 3, since 3 is odd. 

By Lemmas~\ref{sangiovanni}~and~\ref{intervallo}, 
we need to consider only those intervals $[u,v]$ with $\ell(v)-\ell(u)=5$. So $[u,v]\in \{[e,4312], [e,4231], [e,3421], [1243,w_0], [1324,w_0], [2134, w_0]\}$. The intervals $[e,4231]$ and $[1324,w_0]$ are anti-isomorphic: each of them  has two $3$-flipclasses  and each of these two flipclasses has unlabelled support and time-support graphs isomorphic to a 2-crown. By Proposition~\ref{BiB}, each of them has only one increasing path. 

Consider the remaining four intervals. Each of them has one $3$-flipclass $F$, which contains two increasing paths. Moreover, in all these four cases,  $S_F^u$ and $TS_F^u$ are both  isomorphic to a 5-crown.
\end{proof}
\begin{rmk}
Let $F$ be a 3-flipclass of paths from $u$ to $v$. Then
\begin{enumerate}
\item if $|E(F)|\in \{5,6\}$, then $S_{F}^u$ and $TS_F^u$ are both isomorphic to a 3-crown for all $u$ and $v$,
\item if $|E(F)|=3$,  then $S_{F}^u$ and $TS_F^u$ are both isomorphic to a 2-crown for all $u$ and $v$,
\item if $|E(F)|=4$ and $\ell(v) -\ell (u) =5$, then $S_{F}^u$ and $TS_F^u$ are both isomorphic to a 5-crown.
\end{enumerate}
A further analysis completing all cases shows that, if $|E(F)|=4$ and $\ell(v) -\ell (u) =3$, then $S_{F}$  is a 3-crown  except in one case, namely when $F$ corresponds to the flipclass of the path $ \big(2143 \stackrel{(4,2)}{\longrightarrow} 4123   \stackrel{(3,2)}{\longrightarrow} 4132  \stackrel{(2,1)}{\longrightarrow} 4231\big)$ of $\mathfrak S_4$, which is a 4-crown.
\end{rmk}

\begin{rmk}
Note that, for a general flipclass $F$ of paths of length $3$, the support and time-support graphs might be a 5-crown, while this is forbidden for intervals $[u,v]$ with $\ell(v)-\ell(u) = 3$ (see Lemma~\ref{sangiovanni}).
\end{rmk}

\begin{cor}
\label{corh=3}
Let $u,v \in \mathfrak S_n$. The coefficient  $[q^3]\widetilde{R}_{u,v}$ of $q^3$ in $\widetilde{R}_{u,v}$ is

$$[q^3]\widetilde{R}_{u,v} = \sum_{F} c_{t_F},$$

\noindent where $F$ ranges over all $3$-flipclasses of paths from $u$ to $v$, and  $c_{t_F}$ depends only on the $t$-vector $t_F=(t_0,t_1,t_2,t_3)$ of $F$:
\begin{equation*}
 c_{t_F}= \left\{ \begin{array}{ll}
2, & \mbox{if $t_1=5$ (or, equivalently, $t_2=5$)} \\
1, & \mbox{otherwise.} 
\end{array} \right. 
\end{equation*}
\end{cor}
\begin{proof}
By Theorem~\ref{Dyertilde},  the coefficient of $q^3$ of $\widetilde{R}_{u,v}$ is the number of increasing paths of length 3 (with respect to any reflection ordering), which, by Theorem~\ref{corone}, is given by the displayed formula.
\end{proof}

\subsection{$h=4$}

We need the followig result.
\begin{pro}
\label{h<4}
 Let $F$ be a $4$-flipclass in $\mathfrak S_n$. Then $S_F$ and $TS_F$ are isomorphic.
\end{pro}
\begin{proof}
If $F$ is reducible, then the assertion follows by Lemma~\ref{prodotto diretto 2} and Proposition~\ref{flipclass-riducibili}, since $S_G$ and $TS_G$ are isomorphic for all $h$-flipclasses $G$ with $h\leq 3$.

Suppose that $F$ is irreducible. By Theorem~\ref{h+1}, the flipclass $F$ is isomorphic to a $4$-flipclass $G$ of $\mathfrak S_5$.

Let  $u$ (respectively, $v$) be the first (respectively, last) element of the paths in $G$. 
If $\ell(v) - \ell(u) =4$, then the assertion follows by Lemma~\ref{intervallo}.
Hence, we may suppose $\ell(v) - \ell(u) \geq 6$, since $\ell(v) - \ell(u)$ must be even.
The group $\mathfrak S_5$ has 276 intervals  $[u, v]$ with $\ell(v) - \ell(u)$ even and $\geq 6$.  
We checked by computer calculations that the statement holds in all of these cases.
\end{proof}

\begin{rmk}
For $h=5$, there are flipclasses having support graph and time-support graph not isomorphic. For example,
 the paths $\Gamma =  \big(e {\longrightarrow} 2 {\longrightarrow} 32 {\longrightarrow} 232  {\longrightarrow} 2342 {\longrightarrow} 234123121  \big)$  and  $\Delta =   \big(e {\longrightarrow} 232 {\longrightarrow} 3423 {\longrightarrow} 34123  {\longrightarrow} 34123121  {\longrightarrow} 234123121  \big)$ are in the same flipclass $F$. Since $\Gamma$ and $\Delta$ pass through $232$ at time 1 and at time 3, respectively, 
 the support graph $S_F$ and the time-support graph $TS_F$ are not isomorphic.
\end{rmk}

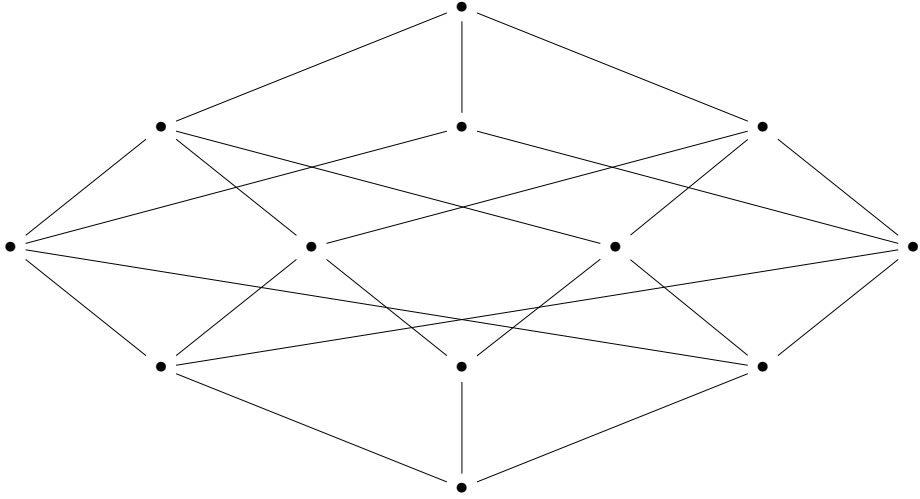
\begin{figure}[h]
    \centering
\scalebox{.8}{    \begin{tikzpicture}
    \node (u) at (0,0) {$\bullet$};
    \node (a1) at (-5,2) {$\bullet$};
    \node (a2) at (0,2) {$\bullet$};
    \node (a3) at (5,2) {$\bullet$};

	\node (b1) at (-7.5,4) {$\bullet$};
    \node (b2) at (-2.5,4) {$\bullet$};
    \node (b3) at (2.55,4) {$\bullet$};
    \node (b4) at (7.5,4) {$\bullet$};    
    
    \node (c1) at (-5,6) {$\bullet$};
    \node (c2) at (0,6) {$\bullet$};
    \node (c3) at (5,6) {$\bullet$};
    \node (v) at (0,8) {$\bullet$};
    \path[-]
    (u) edge (a1)
    (u) edge (a2)
    (u) edge (a3)
    
    (a2) edge (b2)
    (a2) edge (b3)
    (b2) edge (c1)
    (b2) edge (c3)
    (b3) edge (c1)
    (b3) edge (c3)
    
    (a1) edge (b1)
    (a1) edge (b2)
    (a1) edge (b4)
    (a3) edge (b1)
    (a3) edge (b3)
    (a3) edge (b4)
    
    (b1) edge (c1)
    (b4) edge (c3)
    
    (b1) edge (c2)
    (b4) edge (c2)
    
    (c1) edge (v)
    (c2) edge (v)
    (c3) edge (v) 
    ;   
\end{tikzpicture}
}
    \caption{The time-support graph of a flipclass with $t$-vector $(1,3,4,3,1)$}
    \label{espopancia}
\end{figure}

The following result is needed in the proof of Theorem~\ref{teoh=4}. It treats the flipclasses having  $t$-vector equal to $(1,3,4,3,1)$. As we see in the proof of  Theorem~\ref{teoh=4}, when an interval of length $4$ has more than one flipclass and one of them has more than one increasing path, all the other flipclasses have $t$-vector equal to $(1,3,4,3,1)$.  
\begin{pro}
\label{13431}
Let $F$ be a $4$-flipclass having $t$-vector equal to $t_F=(t_0,t_1,t_2,t_3,t_4)=(1,3,4,3,1)$. Then $S_F^u$ and $TS_F^u$ are both isomorphic to the directed graph in Figure~\ref{espopancia}. Every path of $TS_F$ (or, equivalently, $S_F$) is effective. Furthermore, $c(F)=1$.

\end{pro}

\begin{proof}
Let  $u$ (respectively, $v$) be the first (respectively, last) element of the paths in $F$. 
Let $TS^u_{F}=(V,E)$ and $V = \{(u,0), (a_i,1) \textrm { for } i\in \{1,2,3\}, (b_j,2)  \textrm { for } j\in \{1,2,3,4\},  (c_k,3) \textrm { for } k\in \{1,2,3\}, (v,4) \}$. For short, we let $\bar{x} = (x,i)$.

We shall use several times the  property $P_{\bar{x},\bar{y}}$ of Lemma~\ref{P property}.

Given $\bar{x},\bar{y} \in V$, we denote by $A_{\bar{x},\bar{y}}$ the subgraph of $TS_{F}$ given by the vertices and the edges of the paths of $TS^u_{F}$ from $\bar{x}$ to $\bar{y}$. Equivalently, we can define $A_{\bar{x},\bar{y}}$ as the subgraph of $TS^u_{F}$ given by the interval $[\bar{x},\bar{y}]$ in the poset induced by  $TS_{F}$. 

We claim that $A_{\bar{u},\bar{c_i}}$, for $i\in \{1,2,3\}$, is either a 2-crown or a 3-crown. Let us prove the claim.

Let us study, say, $A_{\bar{u},\bar{c_1}}$. Fix a path from $u$ to $v$ in $F$ that passes through the vertex $c_1$, say  $ \big(u {\longrightarrow} a_1 {\longrightarrow} b_1 {\longrightarrow} c_1   {\longrightarrow} v \big)$.
Since every path in the $3$-flipclass $F'$ of the subpath $ \big(u {\longrightarrow} a_1 {\longrightarrow} b_1 {\longrightarrow} c_1  \big)$ can be extended to a path in $F$, we have that $A_{\bar{u},\bar{c_1}}$ contains $TS_{F'}$ as a subgraph. By Theorem~\ref{corone}, since $t_1 = 3$,  the time-support graph $TS^u_{F'}$ is either a 2-crown or a 3-crown. 
A  priori, $TS^u_{F'}$ does not coincide with $A_{\bar{u},\bar{c_1}}$: we show that actually it does.

Suppose first that  $TS^u_{F'}$  is a 2-crown. We may suppose that its vertex set is $\{\bar{u}, \bar{a}_1, \bar{a}_2, \bar{b}_1, \bar{b}_2, \bar{c}_1\}$, and so $\bar{u} \longrightarrow \bar{a}_i \in E$ for $i \in \{1,2\}$, $\bar{a}_i \longrightarrow \bar{b}_j \in E$ for $i \in \{1,2\}$ and $j \in \{1,2\}$, $\bar{b}_i \longrightarrow \bar{c}_1 \in E$ for $i \in \{1,2\}$.
We have $\bar{b}_3 \longrightarrow \bar{c}_1 \notin E$. Indeed, if  $\bar{b}_3 \longrightarrow \bar{c}_1 \in E$, then we could find a path in $F$ of the form $ \big(u {\longrightarrow} a_k {\longrightarrow} b_3 {\longrightarrow} c_1   {\longrightarrow} v \big)$ for some  $k$ in $\{1,2,3\}$. If $k\in \{1,2\}$, then   ($P_{\bar{a}_k,\bar{c}_1}$) would not be true. So $k=3$ but, if we apply the 1-st flip operator to this path, we obtain a path  $ \big(u {\longrightarrow} a_k {\longrightarrow} b_3 {\longrightarrow} c_1   {\longrightarrow} v \big)$ for some  $k$ in $\{1,2\}$, again contradicting ($P_{\bar{a}_k,\bar{c_1}}$). Similarly,  $\bar{b}_4 \longrightarrow \bar{c}_1 \notin E$. Moreover, $a_3 \longrightarrow b_i \notin E$ by ($P_{\bar{u},\bar{b}_i}$), for $i \in \{1,2\}$. Hence,  $A_{\bar{u},\bar{c}_1}$ is a 2-crown.

Now suppose  that  $TS^u_{F'}$  is a 3-crown. We may suppose that its vertex set is the set $\{\bar{u}, \bar{a}_1, \bar{a}_2, \bar{a}_3, \bar{b}_1, \bar{b}_2, \bar{b}_3, \bar{c}_1\}$ and  $\bar{a}_1 \longrightarrow \bar{b}_i \in E$ for $i \in \{1,2\}$, $\bar{a}_2 \longrightarrow \bar{b}_i \in E$ for $i \in \{2,3\}$, and $\bar{a}_3 \longrightarrow \bar{b}_i \in E$ for $i \in \{1,3\}$. We also have   $\bar{u} \longrightarrow \bar{a}_i \in E$ for $i \in \{1,2,3\}$ and  $\bar{b}_i \longrightarrow \bar{c}_1 \in E$ for $i \in \{1,2,3\}$.
We have  $\bar{b}_4 \longrightarrow \bar{c}_1 \notin E$. Indeed, if $\bar{b}_4 \longrightarrow \bar{c}_1 \in E$, then  $\bar{a}_i \longrightarrow \bar{b}_4 \notin E$  by ($P_{\bar{a}_i,\bar{c}_1}$) for $i \in \{1,2,3\}$, and this contradicts the fact that each vertex except $\bar{u}$ has target-valence at least 1.  Moreover, $a_1 \longrightarrow b_3 \notin E$ by ($P_{\bar{a}_1,\bar{c}_1}$), $a_2 \longrightarrow b_1 \notin E$ by ($P_{\bar{a}_2,\bar{c}_1}$), and $a_3 \longrightarrow b_2 \notin E$ by ($P_{\bar{a}_3,\bar{c}_1}$).  Hence,  $A_{\bar{u},\bar{c}_1}$ is a 3-crown.

The claim is proved.

Let $i \in \{1,2,3,4\}$. The element $\bar{b}$ is the source of exactly two edges by Lemma~\ref{almeno2} and ($P_{\bar{b}_i,\bar{v}}$). Since the unique partition of 8 as a sum of 2's and 3's is $8= 3 + 3 + 2$, among the $A_{\bar{u},\bar{c}_i}$, for  $i \in \{1,2,3\}$, there are two 3-crowns and one 2-crown. 

Since the $t$-vector is symmetric, we can use an upside down argument to show that, among the $A_{\bar{a}_i,\bar{v}}$, for  $i \in \{1,2,3\}$, there are two 3-crowns and one 2-crown.

We may suppose that  $A_{\bar{u},\bar{c}_1}$ and  $A_{\bar{u},\bar{c}_3}$ are 3-crowns and $A_{\bar{u},\bar{c}_2}$ is  a 2-crown. 

Let us consider the 2-crown $A_{\bar{u},\bar{c}_2}$. We may suppose that the vertex set of $A_{\bar{u},\bar{c}_2}$  is  $\{\bar{u}, \bar{a}_1,  \bar{a}_3, \bar{b}_1, \bar{b}_4, \bar{c}_2\}$, and so $\bar{u} \longrightarrow \bar{a}_i \in E$ for $i \in \{1,3\}$, $\bar{a}_i \longrightarrow \bar{b}_j \in E$ for $i \in \{1,3\}$, $j \in \{1,4\}$, $\bar{b}_i \longrightarrow \bar{c}_2 \in E$ for $i \in \{1,4\}$, $\bar{a}_2 \longrightarrow \bar{b}_i \notin E$, for $i \in \{1,4\}$, and $\bar{b}_i \longrightarrow \bar{c}_2 \notin E$, for $i \in \{2,3\}$. Hence,  $\bar{a}_2 \longrightarrow \bar{b}_i \in E$, for $i \in \{2,3\}$, by Lemma~\ref{almeno2}. By ($P_{\bar{u}, \bar{b}_i}$), for $i \in \{2,3\}$ and Lemma~\ref{almeno2}, we may suppose  $\bar{a}_1 \longrightarrow \bar{b}_2 \in E$, $\bar{a}_3 \longrightarrow \bar{b}_2 \notin E$, $\bar{a}_3 \longrightarrow \bar{b}_3 \in E$, and $\bar{a}_1 \longrightarrow \bar{b}_3 \notin E$.

 By Lemma~\ref{almeno2},   $\bar{b}_i \longrightarrow \bar{c}_j \in E$, for $i \in \{2,3\}$ and $j \in \{1,3\}$.
 
 By ($P_{ \bar{b}_1,\bar{v}}$) and Lemma~\ref{almeno2}, we may suppose  $\bar{b}_1 \longrightarrow \bar{c}_1 \in E$ and $\bar{b}_1 \longrightarrow \bar{c}_3 \notin E$. 
Since  $A_{\bar{u},\bar{c}_1}$ is a 3-crown, we have  $\bar{b}_4 \longrightarrow \bar{c}_1 \notin E$. Hence,  By Lemma~\ref{almeno2},   $\bar{b}_4 \longrightarrow \bar{c}_3 \in E$.

We have described all edges of $TS^u_{F}$ and obtained the graph in Figure~\ref{espopancia}. By Proposition~\ref{h<4}, it coincides also with $S^u_F$. Hence the first statement holds.

Every path is effective by Lemma~\ref{tutti effettivi}. 

Fix a reflection ordering. It remains to show that there exists a unique increasing path from $u$ to $v$.

Let $\Gamma$ be the lexicographically-first path from $u$ to $v$ in the flipclass $F$, and let $\bar{\Gamma}$ be the corresponding path $\bar{u}$ to $\bar{v}$ in $TS_F$. By Proposition~\ref{BiB},  the path $\Gamma$ is increasing, and so is $\bar{\Gamma}$, since $\Gamma$ and $\bar{\Gamma}$ have the same labels in the same order.

Let $\Delta$ be an increasing path  from $\bar{u}$ to $\bar{v}$ in $TS_F$. We want to show $\Delta = \bar{\Gamma}$. Keep the notations as above.
 Suppose $\bar{c}_i \in \Delta$, for an index $i$ in $\{1,3\}$. Then, by Proposition~\ref{BiB} and Theorem~\ref{corone}, the restriction of $\Delta$ from $\bar{u}$ to $\bar{c}_i$ is the lexicographically-first of 
the 3-crown  $A_{\bar{u},\bar{c}_i}$. Then, the first edge of $\Delta$ coincides with the first edge of $\bar{\Gamma}$. Let $i\in \{1,2,3\}$ be such that this common edge is  $\bar{u} \longrightarrow a_i$. The restrictions of $\Delta$ and $\bar{\Gamma}$ from $\bar{a}_i$ to $\bar{v}$ give two increasing paths of the $3$-flipclass given by $A_{\bar{a}_i,\bar{v}}$, which is either a 2- or a 3-crown.These two restricted paths hence coincide, by Theorem~\ref{corone}.  So $\Delta = \bar{\Gamma}$. 

Suppose $\bar{c}_2 \in \Delta$. Then, $\bar{a}_i \in \Delta$  for an index $i $ in $\{1,3\}$, the subgraph $A_{\bar{a}_i,\bar{v}}$ is a 3-crown and an upside down argument considering the colexicographically-last path shows $\Delta = \bar{\Gamma}$.
\end{proof}

\begin{thm}
\label{teoh=4}
Let $F$ be a $4$-flipclass. The number $c(F)$ depends only on the $t$-vector $t_F=(t_0,t_1,t_2,t_3,t_4)$ of $F$ and is given by the following formula:
\begin{equation*}
 c(F)= \left\{ \begin{array}{ll}
3, & \mbox{if $t_F\in \{(1,5,10,5,1), (1,8,14,8,1)\}$,} \\
2, & \mbox{if $t_1+t_3\geq 12$ and $t_F\neq (1,8,14,8,1)$} \\
1, & \mbox{otherwise.} 
\end{array} \right. 
\end{equation*}
\end{thm}
\begin{proof}
If $F$ is reducible, by Proposition~\ref{flipclass-riducibili} and Theorem~\ref{corone}, the graphs $S_{F}^u$ and $TS_{F}^u$  are both isomorphic to the cartesian product of either
\begin{enumerate}
\item a segment and a $k$-crown for some $k$ in $\{2,3,4,5\}$, or
\item two diamonds.
\end{enumerate}
Then, by Lemma~\ref{prodotto diretto 2}(\ref{quinto'})(\ref{sesto'}), the flipclass $F$ satisfies $c(F)=1$  and  $t_1 + t_3 \leq 10$ except when $TS_F$ is isomorphic to the cartesian product of a segment and a 5-crown, in which case $c(F)=2$ and  and the $t$-vector is equal to $(1,6,10,6,1)$.

Suppose that $F$ is irreducible. By Theorem~\ref{h+1}, the flipclass $F$ is isomorphic to a $4$-flipclass $G$ of $\mathfrak S_5$.

Let  $u$ (respectively, $v$) be the first (respectively, last) element of the paths in $G$. 

 Suppose $\ell(v) - \ell(u) =4$. By Lemma~\ref{intervallo}, one has $c(F)=1$ and that the graph $TS_{F}^u$ is given by the Hasse diagram of the Bruhat interval $[u,v]$. By  \cite[Theorem~3.6]{Hul} for the symmetric group, the $t$-vector $(t_0, t_1, t_2, t_3, t_4)$ of $F$  satisfies 
$$(t_1,t_3) \in \{(3,3), (3,4), (4,3), (4,4), (5,5), (6,5), (5,6)\}.$$ 
Hence $t_1+t_3 <12$; furthermore, also when $(t_1,t_3) = (5,5)$, the $t$-vector $t_F$ cannot be $(1,5,10,5,1)$ by  Lemma~\ref{intervallo}(\ref{555}).

Hence we may suppose $\ell(v) - \ell(u) \geq 6$, since $\ell(v) - \ell(u)$ must be even.

The group $\mathfrak S_5$ has 276 intervals  $[u, v]$ with $\ell(v) - \ell(u)$ even and $\geq 6$. For each of these intervals $[u,v]$, we asked the computer to determine: 
\begin{itemize}
\item the coefficient $[q^4]\widetilde{R}_{u,v}$ (which coincides with the number of increasing paths $i_{u,v}$ from $u$ to $v$ of length 4);
\item the number $n_{u,v}$ of $4$-flipclasses of paths  from $u$ to $v$;
\item for each of these flipclasses, its $t$-vector.
\end{itemize}

The computations show $i_{u,v} \leq 5$. If $i_{u,v} =1$, then  $n_{u,v} = 1$ and the $t$-vector of the only flipclass satisfies $t_1+t_3 \leq 8$.
If $i_{u,v} =2$, then either  $n_{u,v} = 1$ and the $t$-vector $t$ of the only flipclass satisfies both $t_1+t_3 \geq 12$ and $t \neq (1,8,14,8,1)$, or  $n_{u,v} = 2$ and both flipclasses have $t$-vector satisfying $t_1+t_3 \leq 7$. If $i_{u,v} =3$, then  either $n_{u,v}= 3$ and always all three flipclasses have  $(1,3,4,3,1)$ as $t$-vector, or  $n_{u,v} = 1$ and the $t$-vector of the only flipclass is either $(1,5,10,5,1)$ or $(1,8,14,8,1)$.
 If $i_{u,v} =4$, then  $n_{u,v} = 2$ and the two $t$-vectors are $(1,3,4,3,1)$ and $(1,5,10,5,1)$.
 If $i_{u,v} =5$, then  $n_{u,v} = 3$ and the three $t$-vectors are $(1,3,4,3,1)$ with multiplicity 2 and $(1,5,10,5,1)$.

The assertion follows since, by Proposition~\ref{13431}, the flipclasses with   $(1,3,4,3,1)$ as $t$-vector contain exactly one increasing path and, hence, the flipclasses with $t$-vector  $(1,5,10,5,1)$ contain exactly three increasing paths.
\end{proof}

\begin{cor}
\label{corh=4}
Given $u,v \in \mathfrak S_n$, the coefficient  $[q^4]\widetilde{R}_{u,v}$ of $q^4$ in $\widetilde{R}_{u,v}$ is

$$[q^4]\widetilde{R}_{u,v} = \sum_{F} c_{t_F},$$

\noindent where $F$ ranges over all orbits of  $\mathfrak F_{4,u,v}$ and  $c_{t_F}$ depends only on the $t$-vector $t_F=(t_0,t_1,t_2,t_3,t_4)$ of $F$:
\begin{equation*}
 c_{t_F}= \left\{ \begin{array}{ll}
3, & \mbox{if $t_F\in \{(1,5,10,5,1), (1,8,14,8,1)\}$,} \\
2, & \mbox{if $t_1+t_3\geq 12$ and $t_F\neq (1,8,14,8,1)$} \\
1, & \mbox{otherwise.} 
\end{array} \right. 
\end{equation*}
\end{cor}
\begin{proof}
By Theorem~\ref{Dyertilde},  the coefficient of $q^4$ of $\widetilde{R}_{u,v}$ is the number of increasing paths of length 4 from $u$ to $v$, which, by Theorem~\ref{teoh=4}, is given by the displayed formula.
\end{proof}

\subsection{$h=5$}
For $h=5$, the $t$-vector is not a $h$-good invariant (i.e., there are flipclasses with same $t$-vector but different number of increasing paths), but the $\iota$-polynomial is. Actually, the triple given by the $t$-vector, the sequence of the out-degrees of the vertices of the time-support graph of time 1, and the sequence of the in-degrees of the vertices of the time-support graph of time 4 is a $5$-good invariant. This triple is encoded in the $\iota$-polynomial since it is equivalent to $\big(\iota(1,1,t), \frac{\partial \iota}{\partial t}(1,y,0), \frac{1}{(4)!}\frac{\partial^{4} \iota}{\partial t^{4}}(x,1,0) \big)$. For short, given a  flipclass $F$, we let $\operatorname{succ}_F$ and $\operatorname{prec}_F$ denote, respectively, $\frac{\partial \iota_F}{\partial t}(1,y,0)$ and $ \frac{1}{(4)!}\frac{\partial^{4} \iota_F}{\partial t^{4}}(x,1,0)$.

For each $h\in \mathbb N$,  denote by $tVec_h$ the set of vectors occurring as a $t$-vector of an $h$-flipclass  $F$  in the symmetric group $\mathfrak S_n$, for some $n$ in $\mathbb N$. By Corollary~\ref{finiti},  the set  $tVec_h$ is finite, for all $h\in \mathbb N$. The following proposition is the outcome of computer calculations.
\begin{pro}
\label{unionedisgiunta}
The set $tVec_5$ has cardinality 104 and is the disjoint union of the following sets $C_i$, for $i\in[7]$, and $D$:

$$
\begin{array}{lll}
C_1 &=&
\{(1, 3, 5, 5, 3, 1),
 (1, 3, 5, 6, 4, 1),
 (1, 4, 6, 5, 3, 1),
 (1, 4, 7, 7, 4, 1),
 (1, 4, 8, 9, 5, 1),\\
&&
 (1, 4, 9, 10, 5, 1),
 (1, 4, 10, 12, 6, 1),
 (1, 5, 9, 8, 4, 1),
 (1, 5, 10, 9, 4, 1),
 (1, 5, 10, 10, 5, 1),\\
&&
 (1, 5, 10, 11, 6, 1),
 (1, 6, 11, 10, 5, 1),
 (1, 6, 12, 10, 4, 1),
 (1, 6, 13, 13, 6, 1),
 (1, 6, 14, 15, 7, 1),\\
&&
 (1, 6, 14, 16, 8, 1),
 (1, 6, 15, 18, 9, 1),
 (1, 7, 15, 14, 6, 1),
 (1, 7, 17, 17, 7, 1),
 (1, 8, 16, 14, 6, 1),\\
&&
 (1, 9, 18, 15, 6, 1)\}\\
&& \\
C_2  &=&
\{(1, 6, 12, 12, 6, 1),
 (1, 6, 13, 14, 7, 1),
 (1, 7, 14, 13, 6, 1),
 (1, 7, 16, 16, 7, 1),
 (1, 7, 17, 19, 9, 1),\\
&&
 (1, 7, 18, 20, 9, 1),
 (1, 8, 19, 19, 8, 1),
 (1, 8, 20, 20, 8, 1),
 (1, 8, 20, 21, 9, 1),
 (1, 8, 21, 22, 9, 1),\\
&&
 (1, 8, 21, 23, 10, 1),
 (1, 8, 21, 24, 11, 1),
 (1, 8, 22, 24, 10, 1),
 (1, 8, 22, 25, 11, 1),
 (1, 9, 19, 17, 7, 1),\\
&&
 (1, 9, 20, 18, 7, 1),
 (1, 9, 21, 20, 8, 1),
 (1, 9, 22, 21, 8, 1),
 (1, 9, 23, 23, 9, 1),
 (1, 10, 23, 21, 8, 1),\\
&&
 (1, 10, 24, 22, 8, 1),
 (1, 11, 24, 21, 8, 1),
 (1, 11, 25, 22, 8, 1)\}
\end{array}
$$
$$
\begin{array}{lll}
C_3 &=&
\{(1, 6, 15, 15, 6, 1),
 (1, 6, 17, 19, 7, 1),
 (1, 7, 17, 18, 7, 1),
 (1, 7, 18, 17, 7, 1),
 (1, 7, 18, 19, 7, 1),\\
&&
 (1, 7, 18, 20, 7, 1),
 (1, 7, 19, 17, 6, 1),
 (1, 7, 19, 18, 7, 1),
 (1, 7, 19, 21, 8, 1),
 (1, 7, 20, 18, 7, 1),\\
 &&
(1, 8, 21, 19, 7, 1),
 (1, 9, 22, 22, 9, 1),
 (1, 9, 23, 24, 10, 1),
 (1, 9, 24, 26, 11, 1),
 (1, 9, 25, 27, 11, 1),\\
 &&
(1, 10, 24, 23, 9, 1),
 (1, 10, 25, 25, 10, 1),
 (1, 10, 26, 27, 11, 1),
 (1, 10, 27, 27, 10, 1),\\
&&
 (1, 10, 27, 28, 11, 1),
 (1, 10, 27, 29, 12, 1),
 (1, 10, 28, 29, 11, 1),
 (1, 10, 28, 30, 12, 1),\\
&&
 (1, 11, 26, 24, 9, 1),
 (1, 11, 27, 25, 9, 1),
 (1, 11, 27, 26, 10, 1),
 (1, 11, 28, 27, 10, 1),\\
&&
 (1, 11, 29, 28, 10, 1),
 (1, 11, 30, 30, 11, 1),
 (1, 12, 29, 27, 10, 1),
 (1, 12, 30, 28, 10, 1)\}\\
\end{array}
$$
$$
\begin{array}{lll}
C_4 &=&
\{(1, 9, 26, 27, 10, 1),
 (1, 9, 27, 29, 11, 1),
 (1, 10, 27, 26, 9, 1),
 (1, 10, 28, 28, 10, 1),\\
&&
 (1, 10, 29, 30, 11, 1), 
(1, 10, 31, 34, 13, 1),
 (1, 11, 28, 28, 11, 1),
 (1, 11, 29, 27, 9, 1),\\
&&
 (1, 11, 29, 30, 12, 1),
 (1, 11, 30, 29, 10, 1),
 (1, 11, 31, 32, 12, 1),
 (1, 12, 30, 29, 11, 1),\\
&&
 (1, 12, 32, 31, 11, 1),
 (1, 13, 34, 31, 10, 1)\}\\
\end{array}
$$
$$
\begin{array}{lll}
C_5 &=&
\{(1, 9, 24, 25, 9, 1),
 (1, 9, 25, 24, 9, 1),
 (1, 12, 34, 34, 12, 1),
 (1, 12, 35, 36, 13, 1),\\
&&
 (1, 13, 36, 35, 12, 1)\}\\
&&\\
C_6 &=&
\{(1, 8, 27, 28, 9, 1), (1, 9, 28, 27, 8, 1)\}\\
&&\\
C_7& =&
\{(1, 13, 40, 40, 13, 1)\}\\
&&\\
D &=&  \{ (1, 11, 31, 31, 11, 1), (1, 7, 17, 18, 8, 1),  (1, 9, 24, 25, 10, 1),  (1, 8, 18, 17, 7, 1), \\
&&(1, 11, 32, 33, 12, 1),  (1, 10, 25, 24, 9, 1),  (1, 12, 33, 32, 11, 1)\}.
\end{array}
$$
\end{pro}

\begin{defn}
\label{CF}
We define a function $\bar{c}$ from $\mathcal F_5$ to $\mathbb N$. The function $\bar{c}(F)$ depends only on the triple $(\text{$t$-vector $t_F$}, \operatorname{succ_F}(y),  \operatorname{prec_F}(x))$; thus it depends only on  the time-support graph, or, more precisely, on the $\iota$-polynomial of $F$. (Actually, $\bar{c}(F)$ depends only on the $t$-vector except in 7 cases out of 104.)

Let $F$ be a $5$-flipclass. If $t_F\in C_i$, for $i\in[7]$, we let 
$$\bar{c}(F)=i.$$ 

Otherwise, by Proposition~\ref{unionedisgiunta}, we have $t_F\in D$. 

If $t_F= (1, 11, 31, 31, 11, 1)$, then 
$$ \bar{c}(F)= \left\{ \begin{array}{ll}
4, & \mbox{if $\operatorname{succ}_F \in \{4y^4 + 4y^6 + 2y^7 + y^8, 4y^4 + 3y^6 + 4y^7 \}$;} \\
5, & \mbox{if $\operatorname{succ}_F  = 3y^4+ 7y^6+y^8.$} \\
\end{array} \right. $$

If $t_F= (1, 7, 17, 18, 8, 1)$, then 
$$ \bar{c}(F)= \left\{ \begin{array}{ll}
1, & \mbox{if $\operatorname{succ}_F  = y^4+ 6y^5;$}\\
2, & \mbox{if $\operatorname{succ}_F  = 4y^4+ 3y^6.$}\\ 
\end{array} \right. $$

If $t_F= (1, 9, 24, 25, 10, 1)$, then 
$$ \bar{c}(F)= \left\{ \begin{array}{ll}
2, & \mbox{if $\operatorname{succ}_F \in \{3y^4 + 4y^5 + 2y^8, 3y^4 +2y^5 +  2y^6 + 2y^7 , 2y^4 + 4y^5 + y^6 + 2y^7\}$;} \\
3, & \mbox{if  $\operatorname{succ}_F  = 5y^4+ 2y^6 + 2y^8.$} \\
\end{array} \right. $$

If $t_F= (1, 8, 18, 17, 7, 1)$, then 
$$ \bar{c}(F)= \left\{ \begin{array}{ll}
1, & \mbox{if  $\operatorname{succ}_F  = 4y^4+ 4y^5$;} \\
2, & \mbox{if  $\operatorname{succ}_F  = 6y^4+ y^5 + y^7.$} \\
\end{array} \right. $$

If $t_F= (1, 11, 32, 33, 12, 1)$, then 
$$ \bar{c}(F)= \left\{ \begin{array}{ll}
4, & \mbox{if $\operatorname{succ}_F  = 3y^4+ 4y^6 + 4y^7$;} \\
5, & \mbox{if $\operatorname{succ}_F \in \{3y^4 + 6y^6 + 2y^8, 2y^4 +8y^6 + y^8 \}.$} \\
\end{array} \right. $$

If $t_F= (1, 10, 25, 24, 9, 1)$, then 
$$ \bar{c}(F)= \left\{ \begin{array}{ll}
2, & \mbox{if  $\operatorname{succ}_F \in \{5y^4 + y^5 + 3y^6+ y^7, 2y^4 +7y^5 + y^7 , 2y^4 + 6y^5 + 2y^6\}$;} \\
3, & \mbox{if $\operatorname{succ}_F  = 6y^4+ y^5 + 3y^7$.} \\
\end{array} \right. $$

If $t_F= (1, 12, 33, 32, 11, 1)$, then $\bar{c}(F)$ is either $4$ or  $5$. The function $\operatorname{succ}$ does not distinguish the two cases,  but its dual function $\operatorname{prec}$ does. 
$$ \bar{c}(F)= \left\{ \begin{array}{ll}
4, & \mbox{if $\operatorname{prec}_F = 3y^4 + 4y^6 + 4y^7$;} \\
5, & \mbox{if $\operatorname{prec}_F \in \{3y^4 + 6y^6 + 2y^8, 2y^4 +8y^6 + y^8 \}$.} \\
\end{array} \right. $$
The given values above exhaust all possibilities.
\end{defn}

\begin{thm}
\label{teoh=5}
Let $F$ be a 5-flipclass of $\mathfrak S_n$. The number $c(F)$  is given by the function $\bar{c}(F)$ as defined in Definition~\ref{CF}.
\end{thm}
\begin{proof}
Suppose that $F$ is irreducible. By Theorem~\ref{h+1}, the flipclass $F$ is isomorphic to a $5$-flipclass $G$ of $\mathfrak S_6$. By computer analysis, we construct all 5-flipclasses in $\mathfrak S_6$ and, for each of them, its $t$-vector (see Proposition~\ref{unionedisgiunta}) and its number of increasing paths. We see that the number of increasing paths is determined by the $t$-vector except in 7 cases out of 104, when it is determined by the $t$-vector together with the functions $\operatorname{succ}$ and $\operatorname{prec}$. This analysis leads to $c(F)=\bar{c}(F)$.

Suppose that $F$ is reducible. Then, its $t$-vector belongs to $t$-Vec$_1$ $*$ $t$-Vec$_4$ or $t$-Vec$_2$ $*$ $t$-Vec$_3$ (convolution product) by Lemma~\ref{prodotto diretto 2}(\ref{quinto'}), and its number of increasing paths is obtained by the product. We check that all possible convolutions belong to $C_i$ and have $i$ increasing paths except when one of the following two occur:

\begin{enumerate}
\item $F=F_1 * F_2$ for a 1-flipclass $F_1$ and a 4-flipclass $F_2$ with $t_{F_2}=(1,6,11,7,1)$ 
\item $F=F_1 * F_2$ for a 1-flipclass $F_1$ and a 4-flipclass $F_2$ with $t_{F_2}=(1,7,11,6,1)$
\end{enumerate}
In the first case, $c(F)=2$, $t_F=(1,7,17,18,8,1)$ and $\operatorname{succ}_F  = 4y^4+ 3y^6$.
In the second case, $c(F)=2$, $t_F=(1,8,18,17,7,1)$ and $\operatorname{succ}_F  = 6y^4+ y^5 + y^7$.
Hence,  $c(F)=\bar{c}(F)$.
\end{proof}

\begin{cor}
\label{corh=5}
Given $u,v \in \mathfrak S_n$, the coefficient  $[q^5]\widetilde{R}_{u,v}$ of $q^5$ in $\widetilde{R}_{u,v}$ is

$$[q^5]\widetilde{R}_{u,v} = \sum_{F} \bar{c}(F),$$

\noindent where $F$ ranges over all $5$-flipclasses of paths from $u$ to $v$ and  $\bar{c}(F)$ is defined in Definition~\ref{CF}.
\end{cor}
\begin{proof}
By Theorem~\ref{Dyertilde},  the coefficient of $q^5$ of $\widetilde{R}_{u,v}$ is the number of increasing paths of length 5 from $u$ to $v$, which, by Theorem~\ref{teoh=5}, is given by the displayed formula.
\end{proof}

\subsection{$h=6$}
To tackle the case $h=6$, we use Theorem~\ref{strategia}.
\begin{thm}
\label{teoh=6}
The $\iota$-polynomial is an $h$-good invariant for each $h\leq 6$. 
\end{thm}
\begin{proof}
We apply the strategy of Theorem~\ref{strategia}. The $\iota$-polynomial is an $r$-combinatorial invariant, for all $r\in \mathbb N$, and satifies $\iota_{F\ast G}= \iota_F \iota_G$.

Recall that $\mathcal{I}c_h$, $\mathcal{R}c_h$, $\mathcal{I}_h$, and $\mathcal{R}_h$ are finite sets. By computer analysis, we checked that the condition of  Theorem~\ref{strategia} is satisfied.
\end{proof}

To give an idea of the entity of the computation, we note that $|\mathcal F_{6,7}| = 1701056$ and $|\mathcal Ic_6|=4515$.
We also note that the computer calculations show $\mathcal Ic_6 \supset \mathcal Rc_6$.

\begin{cor}
\label{corh=6}
Given $u,v \in \mathfrak S_n$, the coefficient  $[q^6]\widetilde{R}_{u,v}$ of $q^6$ in $\widetilde{R}_{u,v}$ depends only on the $\iota$-polynomials of the $6$-flipclasses of paths from $u$ to $v$.
\end{cor}

\subsection{Consequences}

The first and main consequence of Corollaries~\ref{corh=3},~\ref{corh=4},~\ref{corh=5}, and~\ref{corh=6} is the following.
\begin{thm}
\label{combinatoriainvarianza}
For any $n\in \mathbb N^+$ and for $h\leq 6$,  the coefficient of $q^h$ of the $\widetilde{R}$-polynomials of the symmetric group $\mathfrak S_n$  is a  combinatorial invariant. 
\end{thm}

As a byproduct, we obtain a new proof of the main result of \cite{Inc} for the symmetric group.
\begin{thm}
For any $n\in \mathbb N^+$, the Combinatorial Invariance Conjecture holds for all intervals in the symmetric group $\mathfrak S_n$ up to length 8.
\end{thm}
\begin{proof}
Given $u,v \in \mathfrak S_n$, the polynomial  $\widetilde{R}_{u,v}$ is  monic, has  degree $\ell(v) - \ell(u)$, and satisfies $[q^s]\widetilde{R}_{u,v} = 0$ whenever $s \not\equiv \ell(v) - \ell(u) \pmod2$. Hence the result follows by the results  in this section.
\end{proof}

\section{A conjecture}
We believe that the following conjecture might hold for all $h \in \mathbb N$. The results in the previous section settle it for $h \leq 6$.
\begin{con}
\label{flipclasscongettura}
Let $h\in \mathbb N$. Let $F$ and $G$ be two $h$-flipclasses such that $F^u$ and $G^u$ are isomorphic. Then $c(F) = c(G)$.
\end{con}
Conjecture~\ref{flipclasscongettura} implies the Combinatorial Invariance Conjecture. To give an idea of the rate of growth, we note that  the sequence of the number $|\mathcal F_{h,h+1}|$ of $h$-flipclasses of $\mathfrak S_{h+1}$   is $(4, 4, 50, 1096, 36634, 1701056)$, for $h=1, \ldots  ,6 $.

\bigskip
{\bf Acknowledgments:} 
The computations involved in this work were made using SageMath \cite{Sage}. 
The authors warmly thank Luca Righi for assistance in the use of the HPC cluster of the Department of Mathematics \lq\lq Tullio Levi-Civita\rq\rq of the  Universit\`a di Padova.
The authors are members of the INDAM group GNSAGA. The first author was supported in part by the project of the Universit\`a di Padova [BIRD203834/20]. The first author wishes to thank the Universit\`a Politecnica delle Marche for hospitality.

\end{document}